%% file: main.tex
\documentclass[11pt]{amsart}

\usepackage[top=90pt,bottom=90pt,left=90pt,right=90pt]{geometry}

\usepackage[utf8]{inputenc}
\usepackage[T1]{fontenc}
\usepackage{booktabs} % For formal tables
\usepackage{amsmath} % for correct highlighting only, included with acmart
\usepackage{amsthm} % for correct highlighting only, included with acmart
\usepackage{amssymb}
\usepackage[shortlabels]{enumitem}
\usepackage{mathtools}
\usepackage{leftidx}
\usepackage{nicefrac}
\usepackage{young}
\usepackage{color}
\usepackage[ruled,lined,linesnumbered]{algorithm2e}
\usepackage{algpseudocode}
\usepackage[bookmarksopen=false,pdftex=true,breaklinks=true,%
      hyperindex=true,pdfstartview=FitH,colorlinks=true,%
      pdfpagelabels=true,colorlinks=true,linkcolor=black,%
      citecolor=black,urlcolor=black,hypertexnames=false%
      ]%
   {hyperref}
\usepackage[nameinlink]{cleveref} % better referencing
\usepackage{tikz}
\usepackage{url}
\usetikzlibrary{shapes.geometric}

\theoremstyle{definition}
\newtheorem{theorem}{Theorem}[section]
\newtheorem{proposition}[theorem]{Proposition}
\newtheorem{corollary}[theorem]{Corollary}
\newtheorem{lemma}[theorem]{Lemma}

\newtheorem{question}[theorem]{Question}

\newtheorem{example}[theorem]{Example}
\newtheorem{remark}[theorem]{Remark}

\input{macro22}

\DeclareFontFamily{U}{mathx}{\hyphenchar\font45}
\DeclareFontShape{U}{mathx}{m}{n}{
      <5> <6> <7> <8> <9> <10>
      <10.95> <12> <14.4> <17.28> <20.74> <24.88>
      mathx10
      }{}
\DeclareSymbolFont{mathx}{U}{mathx}{m}{n}
\DeclareFontSubstitution{U}{mathx}{m}{n}
\DeclareMathAccent{\widecheck}{0}{mathx}{"71}

\begin{document}
\nocite*{}
\setlength{\parskip}{0.5\baselineskip}
\setlength{\parindent}{0pt}

\title[Additive and Multiplicative Coinvariant Spaces of Weyl Groups]{Additive and Multiplicative Coinvariant Spaces of Weyl Groups in the Light of Harmonics and Graded Transfer}

\author{Sebastian Debus}
\address{Technische Universität Chemnitz, Fakultät für Mathematik, 09107 Chemnitz, Germany}
\email{sebastian.debus@mathematik.tu-chemnitz.de}

\author{Tobias Metzlaff}
\address{LAAS CNRS, 31400 Toulouse, France}
\email{math@tobiasmetzlaff.com}

\begin{abstract}
The action of a Weyl group on the associated weight lattice induces 
an additive action on the symmetric algebra and 
a multiplicative action on the group algebra of the lattice. 
We show that 
the coinvariant space of the multiplicative action affords the regular representation and 
is isomorphic to a space of multiplicative harmonics, 
which corresponds to existing results for additive coinvariants of reflection groups. 
We then design an algorithm to compute 
a multiplicative coinvariant basis from an additive one. 
The algorithm preserves isotypic decomposition and graded structure 
and enables the study of multiplicative coinvariants 
by integrating combinatorial knowledge from the additive setting. 
We investigate the Weyl groups of type A and C to find 
new explicit equivariant maps and combinatorial structure. 
\end{abstract}

\maketitle

\thispagestyle{empty}

\input{text}

\section*{Acknowledgements}

We are grateful to the anonymous referee for 
the insightful comments and constructive feedback. 
We thank 
Evelyne Hubert, Cordian Riener, Ulrich Thiel and Kirill Zaynullin 
for fruitful discussions and suggestions. 
The work of Tobias Metzlaff has been supported by the project 
``Symbolic Tools in Mathematics and their Application'' 
of the German research foundation (SFB-TRR 195) 
and the Rheinland-Pfalz research initiative (SymbTools). 

\bibliographystyle{alpha}
\bibliography{library.bib}

\end{document}

%% file: macro22.tex
\numberwithin{equation}{section}

\newcommand{\N}{\mathbb{N}}

\newcommand\Z{\mathbb{Z}}
\newcommand\C{\mathbb{C}}

\newcommand\R{\mathbb{R}}

\newcommand{\nops}[1]{\ensuremath{\vert #1  \vert}}

\newcommand{\norm}[1]{\left\lVert#1\right\rVert} %norm
\renewcommand{\det}{\ensuremath{\mathrm{Det}}} %Determinant
\newcommand{\RootA}[1][n-1]{\ensuremath{\mathrm{A}_{#1}}}
\newcommand{\RootB}[1][n]{\ensuremath{\mathrm{B}_{#1}}}
\newcommand{\RootC}[1][n]{\ensuremath{\mathrm{C}_{#1}}}
\newcommand{\RootD}[1][n]{\ensuremath{\mathrm{D}_{#1}}}
\newcommand{\RootE}[1][n]{\ensuremath{\mathrm{E}_{#1}}}
\newcommand{\RootF}[1][n]{\ensuremath{\mathrm{F}_{#1}}}

\newcommand{\weyl}{\ensuremath{\mathcal{W}}} %Weyl group

\newcommand{\roots}{\ensuremath{\rho}}
\newcommand{\Roots}{\ensuremath{\mathrm{R}}}

\newcommand{\Corootlattice}{\ensuremath{\Lambda}}
\newcommand{\sprod}[1]{\ensuremath{\langle#1\rangle}} % scalar product

\newcommand{\fweight}[1]{\ensuremath{\omega_{#1}}}	%fundamental weights
\newcommand{\weight}{\ensuremath{\mu}} %any other weight
\newcommand{\Weights}{\ensuremath{\Omega}} %weightlattice
\newcommand{\CWeights}{\ensuremath{\C[\Weights]}} %weightlattice
\newcommand{\SymWeights}{\ensuremath{\mathrm{Sym}(\Weights)}} %weightlattice

\newcommand{\orb}[1]{\ensuremath{{\theta}_{#1}}} % orbit polynomial
\newcommand{\aorb}[1]{\ensuremath{{\Upsilon}_{#1}}}

\newcommand{\Vor}{\operatorname{Vor}}
\newcommand{\fundom}{\triangle} %Fundamental Domain
\newcommand{\Eulernabla}{\ensuremath{{\nabla}}} % Euler nabla

\def\row#1/#2!{#1_{\IfStrEq{#2}{}{n-1}{#2}} & \dynkin{#1}{#2}\\}

\newcommand{\fundinv}[1]{\ensuremath{\mathcal{S}_\mathrm{#1}}}
\newcommand{\hilbertideal}[1]{\ensuremath{\mathcal{H}_\mathrm{#1}}}
\newcommand{\augmentationideal}[1]{\ensuremath{\mathcal{I}_\mathrm{#1}}}
\newcommand{\coinvariantbasis}[1]{\ensuremath{\mathcal{B}_\mathrm{#1}}}
\newcommand{\GradedWeights}[1]{\ensuremath{\mathcal{G}^*_\mathrm{#1}(\Weights)}}
\newcommand{\harmonics}[1]{\ensuremath{\mathcal{L}_\mathrm{#1}}}

\newcommand{\cartan}{\ensuremath{\mathfrak{h}}}

\DeclareMathOperator{\spn}{span}

\DeclareMathOperator{\Ime}{Im}

\DeclareMathOperator{\Hom}{Hom}

%% file: text.tex
\section*{Introduction}

The linear action of a finite group $\weyl$ on a free $\Z$-module $\Weights$ induces
an additive action on the symmetric algebra $\SymWeights$ and 
a multiplicative action on the group algebra $\CWeights$. 
Of particular interest among these groups are the Weyl groups $\weyl$ of crystallographic root systems that leave the corresponding weight lattice $\Weights$ stable. 
The present article deals with the space of multiplicative coinvariants $\CWeights_\weyl$. 
For Weyl groups, we show that $\CWeights_\weyl$ affords the regular representation of $\weyl$ 
(\Cref{thm:Regual_representation_multiplicative})
and is isomorphic to a space of multiplicative $\weyl$-harmonics 
(\Cref{thm_multiplicative_harmonics_regular_representation}). 
Subsequently, we present an algorithmic method to transform 
an additive symmetry adapted coinvariant basis for $\SymWeights_\weyl$ 
into a multiplicative one via equivariant graded transfer 
(\Cref{thm_CoinvariantTranserAdditiveMultiplicative}, \Cref{alg_CoinvariantBasisTransfer}).

Weight lattices with Weyl group symmetry appear in the representation theory of semi-simple Lie algebras \cite{bourbaki456,bourbaki78}
and provide optimal configurations for a variety of computational problems, 
such as sampling \cite{kunsch05}, 
interpolation \cite{Xu09,Moody2011,HubertSinger2020}, 
and sphere packing \cite{conway1988a,Viazovska17}. 
Algebraic symmetry reduction techniques simplify computations through the elimination of redundant copies of irreducible $\weyl$-modules, 
see for example \cite{gatermann2004symmetry,mustrou21,faugere23}. 

The coinvariant space $\SymWeights_\weyl$ of the additive action, 
which is the quotient of $\SymWeights$ modulo the ideal generated by the homogeneous invariants of positive degree, is well understood. 
A striking result of Chevalley \cite{chevalley1955} states that $\SymWeights_\weyl$ is isomorphic to the regular representation of $\weyl$. 
As such, the coinvariant space gives meaning to the regular representation as a graded $\weyl$-module. 

Finding a symmetry adapted basis for $\SymWeights_\weyl$, 
which respects the isotypic decomposition of the associated group representation, 
has been studied extensively in algebraic combinatorics. 
For symmetric groups (Weyl groups of type $\RootA$), 
a basis can be constructed from Young tableaux and higher Specht polynomials 
\cite{FultonRepresentationTheory,yamada1993higher}. 
Bitableaux combinatorics can furthermore be used to describe coinvariants of 
hyperoctahedral groups (type $\RootB$ and $\RootC$) and 
certain subgroups of index $2$ (type $\RootD$) \cite{ariki1997higher,morita1998higher}. 
Recently, there has been an interest in finding symmetry adapted bases or so-called higher Specht bases for the coinvariant spaces of diagonal actions of 
the symmetric group and Garsia-Procesi modules \cite{GillespieRhoades2021,gillespie2024higher}. 

This extensive study for additive coinvariants raises questions concerning the second action, 
namely the multiplicative action of the Weyl group $\weyl$ on the group algebra $\CWeights$. 
The invariant ring $\CWeights^\weyl$ is, as an algebra, characterized by theorems of Bourbaki and Farkas \cite{bourbaki456,farkas84}. 
Furthermore, the role of multiplicative invariants in geometry was investigated by Demazure \cite{demazure1974} and in following works \cite{garsia1984,griffeth2006rational,Baek2012basicinvariants,huang2014}. 
However, to the knowledge of the authors, the multiplicative coinvariant space $\CWeights_\weyl$ has by far not received the same amount of attention as the additive one. 

The proof that $\CWeights_\weyl$ affords the regular representation 
(\Cref{thm:Regual_representation_multiplicative}), 
begins with a translation of fundamental invariants for $\CWeights^\weyl$ 
by a generic parameter $c$ that allows us to apply a result from \cite{metzlaff2024} 
and follow similar arguments as in the additive setting. 
Throughout, 
we require Steinberg's result that $\CWeights$ 
is a free $\CWeights^\weyl$-module of rank $\nops{\weyl}$ 
whenever $\weyl$ is a Weyl group and $\Weights$ the weight lattice \cite{Steinberg1975}. 

Our notion of multiplicative $\weyl$-harmonics comes with the definition of derivations on $\CWeights$. 
We show that the resulting Jacobian determinant of fundamental invariants is determined by the Weyl denominator, 
which is based on the partial ordering associated to a root system \cite{bourbaki456} and Weyl's character formula \cite{bourbaki78}, 
again very specific for Weyl groups. 
The subsequent proof that also the space of multiplicative $\weyl$-harmonics affords the regular representation (and is therefore isomorphic to the coinvariant space) follows an iterative argument (\Cref{thm_multiplicative_harmonics_regular_representation}). 

A straightforward computation of a multiplicative coinvariant basis, 
for example, by constructing a Gr{\"o}bner bases 
and deriving the normal set, 
does not necessarily reflect these representation theoretic properties, 
see \Cref{example:C_2 normal set}. 
As described above, the additive setting has been extensively studied, 
but a distinguishing difference is that $\CWeights$ 
does not feature an $\N$-grading by polynomial degree compared to $\SymWeights$. 

For a transfer, 
we therefore start from the work of \cite{Baek2012basicinvariants} 
that considers the associated graded algebras. 
The obtained \Cref{alg_CoinvariantBasisTransfer} is based on an equivariant isomorphism between the graded components (\Cref{prop_GradedEquiv}) and an identification of the coinvariants in $\CWeights$ with a quotient of the associated graded algebra (\Cref{prop_GradedCoinvariantIso}). 
The algorithm preserves isotypic decomposition and graded structure. 

Finally, we conclude with case studies for the Weyl groups associated to the Lie algebras of types $\RootA$ and $\RootC$, which feature explicit equivariant homomorphisms, and compare the obtained bases. 

All vector spaces and algebras are defined over the complex numbers $\C$ unless explicitly stated otherwise. 
The nonnegative integers $\{0,1,2,\ldots\}$ are denoted by $\N$.

\section{Additive and Multiplicative Actions}

Let $\Weights$ be a free $\Z$-module of finite rank $n$ 
and $\weyl$ be a finite subgroup of $\mathrm{GL}(\Weights)$. 

Denote by $V = \Weights\otimes_\Z\C$ the vector space generated by $\Weights$. 
The symmetric algebra $\SymWeights$ is the quotient of the tensor algebra $\C\oplus V\oplus(V \otimes V)\oplus\ldots$ 
modulo the relations $\weight_1 \otimes \ldots \otimes \weight_\ell \equiv \weight_{\sigma(1)} \otimes \ldots \otimes \weight_{\sigma(\ell)}$ for all permutations $\sigma \in \mathfrak{S}_\ell$. 
The \textbf{additive action} of $\weyl$ on $\SymWeights$ is the linear action
\begin{equation}
    \weyl\times \SymWeights\to \SymWeights,\,
    (s,\weight_1 \otimes \ldots \otimes \weight_\ell) \mapsto 
    s\cdot (\weight_1 \otimes \ldots \otimes \weight_\ell) := 
    s(\weight_1) \otimes \ldots \otimes s(\weight_\ell) .
\end{equation}
When a $\Z$-basis $\Weights = \Z \fweight{1} \oplus \ldots \oplus \Z \fweight{n}$ is fixed, 
then $\SymWeights = \C[X_1,\ldots,X_n]$ is a polynomial ring with indeterminates $X_i:=\fweight{i}$. 
In particular, we have $s\cdot X_i = s_{1,i}\,X_1+\ldots+s_{n,i}\,X_n$, where $s_{\ell,i}\in\Z$ are the matrix entries of $s \in \weyl$.

The elements of $\SymWeights$, which are invariant under this action, 
form the \textbf{additive invariant ring} $\SymWeights^\weyl$. 
By Hilbert's finiteness theorem, $\SymWeights^\weyl$ is a finitely generated algebra, that is,
$\SymWeights^\weyl = \C[\fundinv{a}]$ for a minimal set of fundamental invariants $\fundinv{a}$. 
We use the subscript ``$\mathrm{a}$'' to indicate objects in the additive setting.

Denote by $\hilbertideal{a} := (\fundinv{a})$ the 
\textbf{additive Hilbert ideal}, that is, 
the ideal in $\SymWeights$ generated by $\fundinv{a}$. 
The \textbf{additive coinvariant space} is the $\weyl$-module
\begin{equation}
    \SymWeights_\weyl
:=  \SymWeights/\hilbertideal{a}.
\end{equation}
We shall now recall a theorem that tells us more about the properties of the objects defined so far. 
A \textbf{pseudoreflection} $s$ is an automorphism on $V$, such that $\mathrm{codim}(\ker(\mathrm{id}_V-s))=1$, 
that is, the fixed point space has codimension $1$. 
A finite group generated by pseudoreflections is called a \textbf{complex reflection group}. 
%In particular, $s$ has eigenvalue $1$ with multiplicity $n-1$ and one other eigenvalue, 
%which is a nontrivial root of unity. 
%If this last eigenvalue is $-1$, then $s$ is called a \textbf{reflection}. 
%If the finite group $\weyl\subseteq\mathrm{GL}(\Weights)$ is generated by reflections over a real subspace of $V$, 
%then it is the \textbf{Weyl group} of some crystallographic root system with \textbf{weight lattice} $\Weights$, 
%see \cite[Ch.~VI,~\S4]{bourbaki456} and \cite[Ch.~9]{kane13}. 

%This article focuses on Weyl groups, 
%but it should be noted that there are also groups generated by pseudoreflections, 
%which leave complex lattices stable, see \cite{Nebe99}. 

Recall that a \textbf{$\weyl$-module} is a vector space $\tilde{V}$ together with 
a group homomorphism $\rho_{\tilde{V}}: \weyl\to\mathrm{GL}(\tilde{V})$, 
called \textbf{representation}. 
Note that $\weyl$ acts on itself and hence on its group algebra $\C[\weyl]$ 
by left multiplication $\weyl\times \C[\weyl] \to \C[\weyl],\, (s,\sum_t c_t\,t) \mapsto \sum_t c_t\,s t$. 
The corresponding group homomorphism $\weyl \to \mathrm{GL}(\C[\weyl])$ 
is the \textbf{regular representation of $\weyl$}. 
A \textbf{$\weyl$-module homomorphism} or \textbf{equivariant map} between two 
$\weyl$-modules $\tilde{V},\tilde{W}$ is a vector space homomorphisms 
$\phi:\tilde{V} \to \tilde{W}$ with $\phi\circ \rho_{\tilde{V}}(s) = \rho_{\tilde{W}}(s) \circ \phi$ 
for all $s\in\weyl$. 
We say that a $\weyl$-module $\tilde{V}$ \textbf{affords the regular representation} if there is a 
$\weyl$-module isomorphism between $\tilde{V}$ and $\C[\weyl]$. 

\begin{theorem}[Shephard-Todd \cite{shephardtodd54}, Chevalley \cite{chevalley1955}]\label{thm:Shephard-Todd-Chevalley}
Let $\weyl$ be a complex reflection group. 
Then the following statements hold. 
\begin{enumerate}
\item $\SymWeights^\weyl$ is a polynomial algebra with Krull-dimension $n$. 
\item $\SymWeights$ is a free $\SymWeights^\weyl$-module of rank $\nops{\weyl}$. 
\item $\SymWeights_\weyl$ affords the regular representation of $\weyl$. 
\end{enumerate}
\end{theorem}

Thanks to statement (1) we can consider the Jacobian determinant $D$ of the map $\C^n \to \C^n$ defined by a set of (homogeneous) fundamental invariants $\fundinv{a}$ with cardinality $n$. 
The vector space $\harmonics{a}$ generated by the partial derivatives of the Jacobian determinant
$\{\partial_{i_1}\ldots\partial_{i_\ell} D\,\vert\,\ell\in\N,\,1\leq i_1\leq\ldots\leq i_\ell \leq n\}$ 
is called the space of \textbf{additive $\weyl$-harmonics}.

\begin{theorem}[Steinberg \cite{SteinbergHarmonicPolynomials}]\label{thm_additive_harmonics}
Let $\weyl$ be a complex reflection group. 
Then $\SymWeights_\weyl$ is isomorphic as a $\weyl$-module to $\harmonics{a}$. 
\end{theorem}

We now introduce the second action. 
The group algebra $\CWeights$ is the commutative algebra with vector space basis 
$\mathfrak{e}^\Weights:=\{\mathfrak{e}^\weight\,\vert\,\weight\in\Weights\}$ and multiplication $\mathfrak{e}^\weight\,\mathfrak{e}^\nu = \mathfrak{e}^{\weight+\nu}$. 
The \textbf{multiplicative action} of $\weyl$ on $\CWeights$ is the linear action
\begin{equation}
    \weyl\times\CWeights\to\CWeights,\,
    (s,\mathfrak{e}^\weight) \mapsto 
    s \cdot \mathfrak{e}^\weight := 
    \mathfrak{e}^{s(\weight)}.
\end{equation}
When a $\Z$-basis $\Weights = \Z \fweight{1} \oplus \ldots \oplus \Z \fweight{n}$ is fixed, 
then $\CWeights = \C[x_1^{\pm 1},\ldots,x_n^{\pm 1}]$ is a Laurent polynomial ring 
with indeterminates $x_i:=\mathfrak{e}^{\fweight{i}}$. 
In particular, we have $s\cdot x_i = x_1^{s_{1,i}}\cdots x_n^{s_{n,i}}$. 

The elements of $\CWeights$, which are invariant under this action, 
form the \textbf{multiplicative invariant ring} $\CWeights^\weyl$. 
This is a finitely generated algebra, that is, $\CWeights^\weyl = \C[\fundinv{m}]$ 
for a minimal set of fundamental invariants $\fundinv{m}$ \cite{lorenz06}. 
We use the subscript ``$\mathrm{m}$'' to indicate objects in the multiplicative setting. 

The notion ``multiplicative'' refers to the fact that $\mathfrak{e}^\Weights$ is a multiplicative group in $\CWeights$, 
where the inverse of $\mathfrak{e}^\weight$ is $\mathfrak{e}^{-\weight}$. 
On the other hand, $V=\spn_\C\Weights$ is an additive group in $\SymWeights$. 

Denote by $\cartan=\Weights\otimes_\Z\R$ the real vector space generated by $\Weights$. 
A \textbf{crystallographic root system} $\Roots$ in $\cartan$ is a subset of $\cartan$, such that
\begin{enumerate}
\item[(R1)] $\Roots$ is finite, does not contain $0$ and generates $\cartan$ as an $\R$-vector space;
\item[(R2)] For every \textbf{root} $\roots\in\Roots$, 
there exists a \textbf{coroot} $\roots^\vee$ inside the dual space $\cartan^\vee$, 
such that $\roots^\vee(\roots)=2$ and the reflection 
$s_\roots:\cartan\to\cartan,\,\weight\mapsto \weight-\roots^\vee(\weight)\,\roots$ leaves $\Roots$ stable;
\item[(R3)] For all $\roots,\tilde{\roots}\in\Roots$, we have $\roots^\vee(\tilde{\roots})\in\Z$. 
\end{enumerate}
The \textbf{weight lattice} of $\Roots$ is the set of all $\weight\in\cartan$, 
such that, for all $\roots\in\Roots$, we have $\roots^\vee(\weight)\in\Z$. 
Furthermore, the \textbf{Weyl group} of $\Roots$ is the group generated by all the reflections $s_\roots$ and leaves the weight lattice stable. 

%The following theorem exists, to the knowledge of the authors, 
%only for Weyl groups acting on weight lattices. 
%By \cite[Ch.~9]{kane13}, a real reflection group\footnote{A real reflection group is a finite group generated by pseudoreflections on a real vector space $\cartan$ with eigenvalues $\pm 1$.}
%that leaves a full-dimensional lattice stable 
%is already the Weyl group of a root system with respective weight lattice. 
%Note that there are also lattices associated to complex root systems \cite{Nebe99}, 
%but we are not aware of a generalization of the following theorem in this context. 

%only for a small subclass of groups generated by pseudoreflections, 
%see also \cite{Nebe99} and \cite[Ch.~9]{kane13}. 

\begin{theorem}[Bourbaki \cite{bourbaki456}, Steinberg \cite{Steinberg1975}, Farkas \cite{farkas84}]\label{thm:Bourbaki}
Let $\weyl$ be the Weyl group of a crystallographic root system with weight lattice $\Weights$. 
Then the following statements hold. 
\begin{enumerate}
\item $\CWeights^\weyl$ is a polynomial algebra with Krull dimension $n$. 
\item $\CWeights$ is a free $\CWeights^\weyl$-module of rank $\nops{\weyl}$. 
\end{enumerate}
\end{theorem}

Note that statements {(1)} and {(2)} correspond to those in \Cref{thm:Shephard-Todd-Chevalley}. 
The hypothesis however is more restrictive, 
since Weyl groups form a proper subclass of the complex reflection groups, see \cite{shephardtodd54}. 
Note that there are also lattices associated to complex root systems \cite{Nebe99}, 
but the authors are not aware of a generalization of \Cref{thm:Bourbaki} in this context. 

\begin{example}\label{example_c2_fundamental_invariants}
The lattice $\Weights=\Z^2$ is the weight lattice of the root system $\RootC[2]$ in $\cartan=\R^2$ with 
fundamental weights $\fweight{1}=e_1,\fweight{2}=e_1+e_2$ and 
Weyl group $\weyl \cong \mathfrak{S}_2 \wr \{\pm 1\}$ of order $8$, 
see  {\cite[Pl. III]{bourbaki456}}. 
We choose between two $\Z$-bases for the weight lattice, namely $\Weights=\Z\,e_1\oplus\Z\,e_2=\Z\fweight{1}\oplus\Z\fweight{2}$, to emphasize that the choice matters for computations. 

In the standard basis $e_1,e_2$, the Weyl group is generated by 
\[
    \begin{pmatrix} 0&1\\1&0 \end{pmatrix}
    \quad \mbox{and} \quad 
    \begin{pmatrix} 1&0\\0&-1 \end{pmatrix} 
\]
with an additive action on $\SymWeights=\C[Y_1,Y_2]$ and a multiplicative action on $\CWeights=\C[y_1^{\pm 1},y_2^{\pm 1}]$. 
The fundamental invariants are $Y_1^2+Y_2^2, Y_1^2\,Y_2^2$ and 
$y_1+y_1^{-1}+y_2+y_2^{-1}, (y_1+y_1^{-1})\,(y_2+y_2^{-1})$ respectively. 

On the other hand, in the basis of fundamental weights $\fweight{1},\fweight{2}$, 
the Weyl group is generated by 
\[
    \begin{pmatrix} -1&0\\1&1 \end{pmatrix}
    \quad \mbox{and} \quad 
    \begin{pmatrix} 1&2\\0&-1 \end{pmatrix}
\]
with an additive action on $\SymWeights=\C[X_1,X_2]$ and a multiplicative action on $\CWeights=\C[x_1^{\pm 1},x_2^{\pm 1}]$. 
We have $X_1=Y_1, X_2=Y_1+Y_2$ and $x_1=y_1, x_2=y_1\,y_2$. 
In particular, the fundamental invariants are $2\,X_1^2-2\,X_1\,X_2+X_2^2, X_1^4-2\,X_1^3\,X_2+X_1^2\,X_2^2$ and 
$x_1+x_1^{-1}+x_1^{-1}x_2+x_1x_2^{-1}, x_2+x_2^{-1}+x_1^2x_2^{-1}+x_1^{-2}x_2$ respectively. 
\end{example}

As a convention, we always use the variables $X_i$ for $\SymWeights$ and $x_i^{\pm 1}$ for $\CWeights$, 
if the basis for $\Weights$ is that of fundamental weights of a crystallographic root system, 
while we use $Y_i$ and $y_i^{\pm 1}$ for %other bases such as 
the standard basis as in \Cref{example_c2_fundamental_invariants}.

\section{Multiplicative Coinvariant Spaces} 

Let $\weyl$ be the Weyl group of a crystallographic root system with weight lattice $\Weights$ and $\fundinv{m} = \{\theta_1,\ldots,\theta_n\}$ 
be a set of fundamental invariants for the multiplicative invariant ring $\CWeights^\weyl$. 
Denote by $\hilbertideal{m} := (\fundinv{m})$ the 
\textbf{multiplicative Hilbert ideal}, that is, 
the ideal in $\CWeights$ generated by $\fundinv{m}$. 
The \textbf{multiplicative coinvariant space} is the $\weyl$-module
\begin{equation}
    \CWeights_\weyl
:=  \CWeights/\hilbertideal{m}.
\end{equation}
In this section, we give a geometric proof that statement {(3)} in \Cref{thm:Shephard-Todd-Chevalley} 
also holds for multiplicative invariants, 
that is, $\CWeights_\weyl$ affords the regular representation. 

\begin{lemma} \label{lem:affine translation of invariant ring}
For $c=(c_1,\ldots,c_n) \in \C^n$, 
the set $\fundinv{m} - c := \{\theta_1 - c_1 , \ldots , \theta_n - c_n\}$  
is also a set of fundamental invariants for $\CWeights^\weyl$.
\end{lemma}
\begin{proof}
We have $\orb{i} - c_i = 1\cdot \orb{i} -c_i\cdot 1 \in \C[\fundinv{m}] = \CWeights^\weyl$ 
and, conversely, $\orb{i} =1\cdot (\orb{i} - c_i) + c_i \cdot 1 \in \C[\fundinv{m} - c]$. 
\end{proof}

Fix a monomial ordering on $\mathfrak{e}^\Weights$ and let $I$ be any ideal in $\CWeights$ with Gr\"{o}bner basis $G$ 
(in practice, one defines $I$ in $\C[y_1,\ldots,y_n,\tilde{y}_1,\ldots,\tilde{y}_n]$ and computes a Gr\"{o}bner basis for $I+(y_i\tilde{y}_i-1)_{1\leq i\leq n}$ with respect to an ordering $y_i\succ \tilde{y}_i$). 
A normal set for $\CWeights / I$ is a finite set of normal forms in $\CWeights$ with respect to $G$, 
such that their equivalence classes form a vector space basis for $\CWeights / I$. 
We do not assume that a normal set consists only of monomials. 

\begin{lemma} \label{lem: quotient size}
For $c \in \C^n$, 
the vector space dimension of $\CWeights/(\fundinv{m} - c)$ is $|\weyl|$ and, 
if $\{h_1,\ldots,h_{|\weyl|}\} \in \CWeights$ is a normal set for 
$\CWeights_\weyl = \CWeights / \hilbertideal{m}$, 
then it is also a normal set for $\CWeights / (\fundinv{m} - c)$ 
and vice versa. 
\end{lemma}
\begin{proof}
By \Cref{thm:Bourbaki} and \Cref{lem:affine translation of invariant ring}, 
there are $h_1,\ldots,h_{|\weyl|} \in \CWeights$, such that
\[
    \CWeights
=   \bigoplus_{\ell=1}^{|\weyl|} h_\ell\cdot \CWeights^\weyl
=   \bigoplus_{\ell=1}^{|\weyl|}h_\ell\cdot \C[\fundinv{m} - c] .    
\]
Hence, for any 
$f=\sum_\ell h_\ell \, g_\ell(\theta_1-c_1,\ldots,\theta_n-c_n) \in \CWeights$ 
with $g_\ell\in\C[z_1,\ldots,z_n]$, 
we have $f - \sum_\ell h_\ell \, g_\ell(0) \in (\fundinv{m} - c)$. 
Since the $h_\ell$ are linearly independent in the quotient, the statement follows.
\end{proof}

\begin{proposition}\label{prop:generic a regular repr}
For $c\in \C^n$ generic, 
the ideal $(\fundinv{m} - c)$ in $\CWeights$ is radical and 
$\CWeights/(\fundinv{m} - c)$ affords the regular representation of $\weyl$. 
\end{proposition}
\begin{proof}
Let $s\in\weyl$ and $z=(z_1,\ldots,z_n) \in (\C\setminus\{0\})^n$
be an element of the algebraic torus. 
Define a left group action $(s,z) \mapsto s\star z$ 
of $\weyl$ on $(\C\setminus\{0\})^n$ by 
\[
    (s \star z)_i
:=  z_1^{\alpha_{i,1}}\cdots z_n^{\alpha_{i,n}}, \quad
    \alpha_{i,j}
\in \Z, \quad
    s^{-1}(\fweight{i}) 
=   \alpha_{i,1} \, \fweight{1} + \ldots + \alpha_{i,n} \, \fweight{n} .
\]
This induces the multiplicative action of $\weyl$ on $\CWeights$, 
that is, $s\cdot f(z) = f(s^{-1} \star z)$. 
Denote by $V_c$ the set of all points in $(\C\setminus\{0\})^n$, 
on which all Laurent polynomials $f\in (\fundinv{m} - c)$ vanish. 
Since the $c_i$ are generic, this set is nonempty. 
By \cite[Prop. 3.1]{metzlaff2024}, 
$V_c$ is the orbit $\weyl\star z$ of a distinguished point 
$\tilde{z} \in (\C\setminus\{0\})^n$. 
In particular, the orbit has full length, 
that is, $\nops{\weyl}=\nops{\weyl\star \tilde{z}}$, 
and $(\fundinv{m} - c)$ is zero-dimensional. 
We have 
\[
    \dim_{\C} (\CWeights/(\fundinv{m} - c)) 
\geq\nops{V_c} 
=   \nops{\weyl\star \tilde{z}} 
=   \nops{\weyl}
\]
and equality holds if and only if $(\fundinv{m} - c)$ is radical. 
With \Cref{lem: quotient size}, 
we have $\dim_{\C} (\CWeights/(\fundinv{m} - c)) = \nops{\weyl}$, 
which shows that $(\fundinv{m} - c)$ is indeed radical. 

Now, 
the coordinate ring $\CWeights/(\fundinv{m} - c)$ of $V_c$ 
corresponds to the algebra $\C[\weyl\star \tilde{z}]$ 
of regular functions on $\weyl\star \tilde{z}$. 
We claim that 
$\Hom_\weyl(\weyl\star \tilde{z}, \C)$ 
and $\C[\weyl\star \tilde{z}]$ 
are isomorphic as $\weyl$-modules. 
For $z \in \weyl(\tilde{z})$, 
let $\iota_{z} \in \C[\weyl\star \tilde{z}]$ 
denote the indicator function of $z$, 
that is, $\iota_{z}(z)=1$ and 
$\iota_{z}(z')=0$ for all 
$z' \in \weyl\star \tilde{z} \setminus\{ z\}$. 
For any $s \in \weyl$, we have 
$s \cdot \iota_{z}(z') = \iota_{z}(s^{-1} \star z') = \iota_{s \star z}(z') $. 
This shows that 
\[
    \Hom_\weyl(\weyl\star \tilde{z}, \C) \to \C[\weyl\star \tilde{z}] , \, 
    \iota_{z} \mapsto z
\]
is a $\weyl$-module homomorphism and bijective. 
Hence, the claim is true and the statement follows from the observation that $\C[\weyl] $ and $ \C[\weyl\star \tilde{z}]$ are isomorphic as $\weyl$-modules.
\end{proof}

\begin{theorem}\label{thm:Regual_representation_multiplicative}
The multiplicative coinvariant space $\CWeights_\weyl$ affords the regular representation of $\weyl$. 
\end{theorem}
\begin{proof}
Let $c\in\C^n$ be a generic point. 
We prove that the vector spaces $\CWeights_\weyl$ and 
$\CWeights / (\fundinv{m} - c)$ are isomorphic as $\weyl$-modules. 
Since \Cref{prop:generic a regular repr} implies 
$\CWeights / (\fundinv{m} - c) \cong \C[\weyl]$, 
this will show the claim. 

By \Cref{lem: quotient size}, 
there exists a normal set 
$\{h_1,\ldots,h_{|\weyl|}\} \subset \CWeights$ of 
both $\hilbertideal{m} = (\fundinv{m} - 0)$ and $(\fundinv{m} - c)$, 
such that 
\[
    \mathcal{R}_c
:=  \CWeights/(\fundinv{m} - c) 
=   \operatorname{span}\{[h_1]_{\mathcal{R}_c},\ldots,[h_{|\weyl|}]_{\mathcal{R}_c}\} 
\]
with $\mathcal{R}_0 = \CWeights_\weyl$ and thus
\[
    \Phi_c : 
    \mathcal{R}_c \to \CWeights_\weyl, \, 
    \sum_{\ell=1}^{|\weyl|} \lambda_\ell \, [h_\ell]_{\mathcal{R}_c} 
    \mapsto 
    \sum_{\ell=1}^{|\weyl|} \lambda_\ell \, [h_\ell]_{\mathcal{R}_0} \]
is a well-defined vector space isomorphism. 
Moreover, the ideals $\hilbertideal{m}$ and $(\fundinv{m} - c)$ are $\weyl$-stable. 
Since $\CWeights$ is a free $\CWeights^{\weyl}$-module, 
we have $s \cdot h_\ell \in \operatorname{span}\{h_1,\ldots,h_{|\weyl|}\}$ for all $s \in \weyl$. 
Thus, $\Phi_c$ is a $\weyl$-module isomorphism. 
\end{proof}

A multiplicative coinvariant basis can for example be computed as follows. 

\begin{algorithm}
\caption{Multiplicative Coinvariant Bases}\label{alg_MultiplicativeCoinvariantBasis}
\KwData{$\fundinv{m} \subseteq \C[x_1^{\pm 1},\ldots,x_n^{\pm 1}]$ 
a set of fundamental invariants for $\CWeights^\weyl$}
\KwResult{$\coinvariantbasis{m} \subseteq \C[x_1^{\pm 1},\ldots,x_n^{\pm 1}]$ 
a normal set for $\CWeights_\weyl$}
$\hilbertideal{m}=(\fundinv{m})$\;
$G=\mathrm{GB}(\hilbertideal{m})$ \Comment{Gr\"{o}bner basis}\;
$N=\mathrm{NS}(G)$ \Comment{normal set}\;
\textbf{return} $\coinvariantbasis{m}=N$\;
\end{algorithm}

\begin{example}\label{example:C_2 normal set}
Recall from  {\Cref{example_c2_fundamental_invariants}} 
that the multiplicative fundamental invariants for $\RootC[2]$ 
in the basis of fundamental weights are the orbit polynomials 
$\theta_1 = x_1+x_1^{-1}+x_1^{-1}x_2+x_1x_2^{-1}$ and 
$\theta_2 = x_2+x_2^{-1}+x_1^{-2}x_2+x_1^2x_2^{-1}$. 
A normal set of the Hilbert ideal $\hilbertideal{m} = (\theta_1,\theta_2)$ 
with respect to the graded reverse lexicographical ordering 
$x_1\succ x_2 \succ x_1^{-1}\succ x_2^{-1}$ is 
\[  
    \left\{
    1, \, x_1, \, x_1x_2^{-1}, \, x_2, \, x_1^{-1}, \, x_1^{-2},\, x_1^{-1}x_2^{-1}, \, x_2^{-1}
    \right\}. 
\]
Note that this normal set consists of monomials but does not tell us anything about the $\weyl$-module structure of $\CWeights_\weyl$.
We shall treat this circumstance in \Cref{sec_associated_graded_equivariants,sec_graded_transfer}. 
\end{example}

\section{Multiplicative $\weyl$-Harmonics}\label{Sec:Weyl Harmonics}

Let $\weyl$ be a Weyl group of a crystallographic root system with fundamental weights $\fweight{i}$
and $\Weights=\Z\fweight{1}\oplus\ldots\oplus\Z\fweight{n}$ be the weight lattice. 
In analogy to \Cref{thm_additive_harmonics}, 
we show in the present section that the multiplicative coinvariant space 
can also be understood as a space of harmonics. 

For this, we define the \textbf{Euler derivation} $\partial_i$ as the linear operator on $\CWeights$ with 
$\partial_i \,x^\alpha= \alpha_i\,x^\alpha$ for $\alpha\in\Z^n$. 
The associated nabla operator is denoted by $\Eulernabla:= [ \partial_1 ,\ldots, \partial_n ]^t$. 
This definition is motivated by the derivative of the exponential function, 
see also \cite{HoffmanWithers,MuntheKaas2012}. 
Notation is the same as for the derivation on $\SymWeights$ before \Cref{thm_additive_harmonics} but clear form the context. 

Recall from \cite[Ch.~VI,~\S3,~Thm.~1]{bourbaki456} 
that the multiplicative invariant ring $\CWeights^\weyl$ of a Weyl group is generated as an algebra by the orbit polynomials 
$\orb{i} := \frac{1}{\nops{\weyl}}\sum_{s\in\weyl} \mathfrak{e}^{s(\fweight{i})}$, 
where $\fweight{i}$ are the fundamental weights. 
Denote the corresponding Jacobian matrix by $J := [ \Eulernabla \, \orb{1} \vert \ldots \vert \Eulernabla \, \orb{n} ] \in (\CWeights)^{n\times n}$. 

\begin{proposition}\label{prop:Weyl determinant and denominator}
Let $\delta:=\fweight{1}+\ldots+\fweight{n} \in\Weights$ be the sum of the fundamental weights. 
The Jacobian determinant $\det(J)$ is, up to a nonzero scalar, equal to 
\[  
    \aorb{\delta} 
:=  \frac{1}{\nops{\weyl}} \sum\limits_{s\in\weyl} \det(s) \, \mathfrak{e}^{s(\delta)} 
\in \CWeights. 
\]
\end{proposition}

A Laurent polynomial $f \in \CWeights$ satisfying 
$s \cdot f = \det (s) f$ whenever $s \in \weyl$ 
is called \textbf{anti-invariant}. 
By Weyl's character formula \cite[Ch.~VIII,~\S9,~Thm.~1]{bourbaki78}, 
the space of all anti-invariants is a free $\CWeights^\weyl$-module 
of rank $1$ generated by $\aorb{\delta}$, 
which is therefore also called the \textbf{Weyl denominator}. 

\begin{proof}
For $s\in\weyl$ and $z\in(\C\setminus\{0\})^n$, one observes that $J(s^{-1}\star z) = (s_{i,j})_{1\leq i,j\leq n}\,J(z)$, where the right hand side is given by matrix multiplication. 
Hence, $\det(J)$ is anti-invariant and thus there exists a unique $f\in\CWeights^\weyl$, 
such that $\det(J) = f\,\aorb{\delta}$. 
Define a partial ordering on $\Weights$ by 
$\weight \prec \fweight{}$ if and only if $0\neq \fweight{} - \weight$ 
is a sum of positive roots of the corresponding root system. 
We have 
\[
	J_{ij}
=	\partial_i \orb{j}
=	\dfrac{\partial_i}{\nops{\weyl}} \sum\limits_{s \in \weyl} \mathfrak{e}^{s(\fweight{j})}
=	\dfrac{\partial_i}{\nops{\weyl( \fweight{j} )}} \left( \mathfrak{e}^{\fweight{j}} 
    + \sum\limits_{\weight\prec \fweight{j}} \tilde{c}_\weight\,\mathfrak{e}^\weight \right) ,
\]
where $\tilde{c}_\weight\in\{0,1\}$. 
The diagonal elements of the matrix $J$ have highest monomial $\mathfrak{e}^{\fweight{i}}$ in this ordering, 
while the non-diagonal entries have highest monomial $\mathfrak{e}^\weight$ for some $\weight\prec \fweight{j}$. 
Hence, there exist $c_\weight \in \Z$ such that 
\begin{align*}
	f\,\aorb{\delta}
\, = \, &	\det(J)
\, = \, 	\dfrac{1}{\prod\limits_{i=1}^n {\nops{\weyl(\fweight{j})}}} \, \left( \mathfrak{e}^{\delta} 
            + \sum\limits_{\weight \prec \delta} c_\weight\,\mathfrak{e}^\weight \right) .
\end{align*}
By comparison of leading monomials, $f$ must be a scalar. 
\end{proof}

We call the vector space $\harmonics{m}$ generated by 
$\{\partial_{i_1}\ldots\partial_{i_\ell} \det(J)\,\vert\,\ell\in\N,\,1\leq i_1\leq \ldots\leq i_\ell \leq n\}$ 
the space of \textbf{multiplicative $\weyl$-harmonics}. 

\begin{theorem}\label{thm_multiplicative_harmonics_regular_representation}
The space of multiplicative $\weyl$-harmonics affords the regular representation of $\weyl$. 
\end{theorem}
\begin{proof}
By \Cref{prop:Weyl determinant and denominator} the Jacobian determinant of $\weyl$ is, up to a scalar, equal to the Weyl denominator $\aorb{\delta}$. 
Since $\delta=\fweight{1}+\ldots+\fweight{n}$ is the sum of all the fundamental weights, 
it is contained in the interior of the fundamental Weyl chamber. 
But $\weyl$ acts transitively on the set of all chambers and thus $\nops{\weyl( \delta)} = \nops{\weyl}$ and $\aorb{\delta}$ is a Laurent polynomial with $\nops{\weyl}$-many terms $c_\weight\,x^\weight$, 
where $\weight \in \weyl( e)$ and $c_\weight = \pm 1/\nops{\weyl}$. 
From the definition of $\partial_i$, it follows that $\harmonics{m}$ is a subspace of $L:= \spn_{\C}\{ \mathfrak{e}^\weight\,\vert\,\weight\in\weyl(\delta)\}$. 
Since $\nops{\weyl (\delta)}=\nops{\weyl}$, $L$ affords the regular representation of $\weyl$. 

We show that $\harmonics{m}=L$ with a recursive argument. 
Let $W_0:=\{\weight=\alpha_1\,\fweight{1}+\ldots+\alpha_n\,\fweight{n} \in \weyl (\delta) \, \vert \, \forall \, 1\leq i \leq n: \, \alpha_i \neq 0 \}$ be the set of all monomials in $\aorb{\delta}$ with full support. 
First, we want to show that any monomial occurring in $\aorb{\delta}$ is contained in $\harmonics{m}$. 
We have $f:=\partial_1\ldots\partial_n \aorb{\delta}=\sum_{\weight \in W_0} c_\weight \mathfrak{e}^\weight \in \harmonics{m}$ for some $c_\weight \in \C\setminus\{0\}$. 
We order all the elements in $W_0$ totally $\weight^{(1)} \prec \weight^{(2)}  \prec \ldots \prec \weight^{(n_0)}$ and show that $\mathfrak{e}^{\weight^{(1)}} \in \harmonics{m}$. 
Let $\weight^{(\ell)} = \alpha_1^\ell\,\fweight{1} + \ldots + \alpha_n^\ell \, \fweight{n}$ and $1\leq j\leq n$ 
with $\weight_j^{(1)} \neq \weight_j^{(2)}$. 
Consider $\partial_j f- \alpha_j^2 \, f = \sum_{i=1}^{n_0} c_{\weight^{(i)}}(\alpha_j^{(i)}-\alpha_j^{(2)})\mathfrak{e}^{\weight^{(i)}} \in \harmonics{m}$.
By iterating this procedure we find $\mathfrak{e}^{\weight^{(1)}} \in \harmonics{m}$ and eventually that 
$\mathfrak{e}^{\weight^{(i)}} \in \harmonics{m}$ for any $1 \leq i \leq n_0$. 
We can repeat the same strategy by considering 
$f_1 := \sum_{\weight \in \weyl (\delta) \setminus W_0} c_\weight \mathfrak{e}^\weight$ 
and an inclusion-wise maximal support set $W_1$ in $f_1$ of cardinality $n_{1}$. 
After finitely many steps we obtain $\harmonics{m}=L$. 
\end{proof}

Concluding with \Cref{thm:Regual_representation_multiplicative}, 
the multiplicative coinvariant space $\CWeights_\weyl$ is, as a $\weyl$-module, 
isomorphic to the space of multiplicative $\weyl$-harmonics. 

\begin{remark}
For the action of a pseudoreflection group on polynomials and, 
in particular, for the additive action of a Weyl group on $\SymWeights = \C[X_1,\ldots,X_n]$, 
the additive $\weyl$-harmonic polynomials can equivalently be defined as 
the orthogonal complement of $\hilbertideal{a}$ with respect to the inner product 
$\sprod{\cdot,\cdot} : \SymWeights \times \SymWeights \to \C, \sprod{f,g}:=(\partial_f (g))(0)$, where 
$\partial_f := \sum_{\alpha}c_\alpha \partial_1^{\alpha_1} \ldots \partial_n^{\alpha_n}$ 
if $f=\sum_\alpha c_\alpha X^\alpha$, 
see for example {\cite{swanson}}.
The name ``harmonics'' originates from this construction.
As a consequence we have that the tensor product over $\C$ of the additive coinvariant space $\SymWeights/\hilbertideal{a}$ with the space of additive $\weyl$-harmonics is isomorphic to $\SymWeights$. 

Unfortunately, this does not hold in the multiplicative setting. 
For instance, for $\weyl$ the Weyl group of the root system $\RootA[2]$, we found that 
\[
\CWeights^\weyl \otimes_\C 
\spn_{\C}\{\partial_{1}^k\,\partial_{2}^\ell \, \aorb{\delta}\,\vert\,k, \ell\in\N\}
\] 
is a proper subset of $\CWeights$, where the Weyl denominator $\aorb{\delta}$ is equal to the Laurent polynomial $g$ 
in  {\Cref{table:A_2 additive and multiplicative coinvariants}} up to a scalar. 
\end{remark}

\section{Associated Graded Equivariants}\label{sec_associated_graded_equivariants}

In the previous two sections, we have described the multiplicative coinvariant space in terms of representation-theoretic properties. 
From a computational point of view however, 
the obtained basis for $\CWeights_\weyl$ from \Cref{alg_MultiplicativeCoinvariantBasis}
does not reflect those properties. 
Meanwhile, many combinatorial works facilitate the computation in the additive setting, 
see \cite{yamada1993higher,ariki1997higher,morita1998higher}. 
Our next goal in this and the following section is therefore to provide an algorithm, 
that allows the transfer of an additive coinvariant basis to a multiplicative one. 
This makes it possible to obtain a symmetry adapted basis for $\CWeights_\weyl$ from one for $\SymWeights_\weyl$. 

To commence, we follow \cite{Baek2012basicinvariants} and consider the algebra homomorphisms
\begin{equation}
    \SymWeights \to \C,\,X_i\mapsto 0
    \quad \mbox{and} \quad
    \CWeights \to \C,\,x_i\mapsto 1
\end{equation}
with kernels $\augmentationideal{a} = (X_i\,\vert\,1\leq i\leq n)$ and 
$\augmentationideal{m} = (1-x_i^{-1}\,\vert\,1\leq i\leq n)$ respectively. 

\begin{lemma}
For a nonnegative integer $d\in\N$, 
\[
    \phi_d: \SymWeights / \augmentationideal{a}^{d+1} \to \CWeights / \augmentationideal{m}^{d+1} ,\, [X_i] \mapsto [1 - x_i^{-1}]
\]
is an algebra isomorphism. 
\end{lemma}
\begin{proof}
The inverse homomorphism is defined by $\phi_d^{-1}([x_i^{-1}]) = [1-X_i]$ and $\phi_d^{-1}([x_i]) = [1+X_i+X_i^2+\ldots+X_i^d]$ by geometric series expansion modulo $\augmentationideal{a}^{d+1}$. 
\end{proof}

Note that $\phi_d$ can be lifted to an algebra homomorphism $\phi: \SymWeights \to \CWeights$. 
Conversely, any $\psi: \CWeights \to \SymWeights$ with
$1=\psi(1)=\psi(x_i\,x_i^{-1})=\psi(x_i)\,\psi(x_i^{-1})$ forces $\psi(x_i)$ to be a unit in $\SymWeights$, 
that is, $\psi(x_i)\in\C\setminus\{0\}$. 
Hence, $\phi$ is not invertible. 

Instead, we consider the associated graded algebras \cite{Bruns_Herzog_1998}
\begin{equation}
    \GradedWeights{a} := \bigoplus\limits_{d\in\N} \augmentationideal{a}^d / \augmentationideal{a}^{d+1}
    \quad \mbox{and} \quad
    \GradedWeights{m} := \bigoplus\limits_{d\in\N} \augmentationideal{m}^d / \augmentationideal{m}^{d+1} .
\end{equation}
The quotient $\augmentationideal{a}^d / \augmentationideal{a}^{d+1}$ 
is isomorphic to the vector space $\SymWeights_d$ of homogeneous polynomials of degree $d$. 
In particular, we have $\GradedWeights{a}\cong\SymWeights$ as graded algebras.  

For $f\in\CWeights$ on the other hand, we denote by $f^*:=[f_0]+[f_1]+[f_2]+\ldots$ the associated graded element in $\GradedWeights{m}$, 
where $f_{0}$ is the normal form of $f$ with respect to $\augmentationideal{m}$ and iteratively 
$f_d$ is the normal form of $f-f_{d-1}-f_{d-2}-\ldots-f_0$ 
with respect to $\augmentationideal{m}^{d+1}$ so that $[f_d] \in \augmentationideal{m}^d / \augmentationideal{m}^{d+1}$. 
We have $f=f_0+f_1+f_2+\ldots$ and, since $\GradedWeights{m}$ is Noetherian, this iteration stops at some point. 

\begin{example}
 {(1)} Consider the homogeneous polynomial 
\[
    F
=   X_1\,X_2-X_1^2\in\C[X_1,X_2]_2 \cong \augmentationideal{a}^2 / \augmentationideal{a}^3 .
\]
Its image under $X_i \mapsto 1-x_i^{-1}$ decomposes into 
\begin{align*}
    f
=&  \frac{1}{x_1} - \frac{1}{x_1^2} - \frac{1}{x_2} + \frac{1}{x_1\,x_2} \\
=&  \underbrace{\left(\frac{1}{x_1\,x_2} - x_1 - \frac{2}{x_1} - \frac{1}{x_2} + 3\right)}_{
    =f_2}
+   \underbrace{\left(-\frac{1}{x_1^2} + x_1 + \frac{3}{x_1} - 3\right)}_{
    =f_3} 
\end{align*}
with $[f_2]\in\augmentationideal{m}^2 / \augmentationideal{m}^3$ and $[f_3]\in\augmentationideal{m}^3 / \augmentationideal{m}^4$. 

 {(2)} Now, let $\weyl$ be a Weyl group of the root system $\RootC[2]$ as in \Cref{example:C_2 normal set}. 
Since $\weyl(\fweight{1})=\{\fweight{1},\fweight{1}-\fweight{2},\fweight{2}-\fweight{1},-\fweight{1}\}$, we have 
\begin{align*}
    \C[x_1^{\pm 1},x_2^{\pm 1}]^\weyl
\ni&\,x_1 + \frac{x_1}{x_2} + \frac{x_2}{x_1} + \frac{1}{x_1} - 4 \\
=&  \underbrace{\left(\frac{-2}{x_1\,x_2} + 2\,x_1 + x_2 + \frac{4}{x_1} + \frac{3}{x_2} - 8 \right)}_{
    =f_2} \\
&+  \underbrace{\left(\frac{x_1}{x_2} + \frac{x_2}{x_1} + \frac{2}{x_1\,x_2} - x_1 - x_2 - \frac{3}{x_1} - \frac{3}{x_2} + 4\right)}_{
    =f_3}. 
\end{align*}
Neither of the two summands is stable under the $\weyl$-transformation 
$s_1 : \, ( x_1 , x_2 ) \mapsto ( x_1^{-1} \, x_2 , x_2 )$. 
\end{example}

We conclude from the example that $\phi_d$ preserves neither homogeneity nor equivariance. 
We explain now how working over the associated graded algebras instead solves this circumstance. 

\begin{proposition}[\cite{Baek2012basicinvariants,CalmesPetrovZainoulline2013}]\label{prop_GradedEquiv}
The quotients 
$\SymWeights_d \cong \augmentationideal{a}^d / \augmentationideal{a}^{d+1}$ and $\augmentationideal{m}^d / \augmentationideal{m}^{d+1}$ 
are $\weyl$-stable under the induced additive and multiplicative action respectively. 
Furthermore, the map
\[
    \phi^{(d)} : 
    \begin{array}[t]{ccc}
    \SymWeights_d                  &   \to     & 
    \augmentationideal{m}^d / \augmentationideal{m}^{d+1}, \\
    \prod\limits_{j=1}^d X_{i_j}   &   \mapsto & 
    \left[ \prod\limits_{j=1}^d (1-x_{i_j}^{-1}) \right]
    \end{array}
\]
is a $\weyl$-module isomorphism. 
\end{proposition}
\begin{proof}
Let $s\in\weyl$ with matrix entries $s_{\ell,i} \in \Z$. 
Then 
\[
    \phi^{(d)}\left(s\cdot\prod\limits_{j=1}^d X_{i_j}\right)
=   \phi^{(d)}\left(\prod\limits_{j=1}^d \sum\limits_{\ell=1}^n s_{\ell,i_j} \,X_\ell\right)
=   \left[\prod\limits_{j=1}^d \sum\limits_{\ell=1}^n s_{\ell,i_j} \,(1-x_\ell^{-1})\right].
\]
Since $(1-x^{-\alpha})\,(1-x^{-\beta}) = (1-x^{-\alpha})+(1-x^{-\beta})-(1-x^{-\alpha-\beta})$ for any $\alpha,\beta\in\Z^n$
and $ [1-x_\ell^{-1}] = [-(1-x_\ell)]$ in $\augmentationideal{m}/\augmentationideal{m}^2$, 
we have $[s_{\ell,i_j}(1-x_\ell^{-1})]=[ 1-x_\ell^{-s_{\ell,i_j}}]$ in $\augmentationideal{m}/\augmentationideal{m}^2$ 
and consequently 
\[
    \left[\prod\limits_{j=1}^d \sum\limits_{\ell=1}^n s_{\ell,i_j} \,(1-x_\ell^{-1})\right]
=   \left[\prod\limits_{j=1}^d \sum\limits_{\ell=1}^n \left(1-x_\ell^{-s_{\ell,i_j}}\right)\right]
\]
in $\augmentationideal{m}^{d}/\augmentationideal{m}^{d+1}$. 
Thus, 
\[
    \phi^{(d)}\left(s\cdot\prod\limits_{j=1}^d X_{i_j}\right)
=   \left[\prod\limits_{j=1}^d \sum\limits_{\ell=1}^n \left(1-x_\ell^{-s_{\ell,i_j}}\right)\right]
=   s \cdot \phi^{(d)}\left(\prod\limits_{j=1}^d X_{i_j}\right) . 
\]
We conclude that $\phi^{(d)}$ is well-defined and $\weyl$-equivariant. 
For bijectiveness, note that 
$\{ [ (1-x_{i_1}^{-1}) \cdots (1-x_{i_d}^{-1}) ]\,\vert\,1\leq i_1 \leq \ldots\leq i_d\leq n \}$ 
is a vector space basis for $\augmentationideal{m}^{d}/\augmentationideal{m}^{d+1}$. 
\end{proof}

In particular, 
\begin{equation}
    \phi^*
=   \bigoplus\limits_{d\in\N} \phi^{(d)}:\SymWeights \to \GradedWeights{m} 
\end{equation}
is a $\weyl$-equivariant graded algebra isomorphism.  

\section{Graded Transfer of Coinvariants}\label{sec_graded_transfer}

In this section, we show how to preserve group symmetry between additive and multiplicative coinvariants. 

Let $\weyl$ be a Weyl group of a crystallographic root system with fundamental weights $\fweight{i}$
and $\Weights=\Z\fweight{1}\oplus\ldots\oplus\Z\fweight{n}$ be the weight lattice. 
As fundamental invariants, we consider the orbit polynomials 
$\orb{i} := \frac{1}{\nops{\weyl}}\sum_{s\in\weyl} \mathfrak{e}^{s(\fweight{i})}$. 
Note that $\orb{i}-1\in\augmentationideal{m}$. 
By \Cref{lem:affine translation of invariant ring}, the set
$\{\orb{1}-1,\ldots,\orb{n}-1\}\subseteq\augmentationideal{m}\cap\CWeights^\weyl$ 
also generates $\CWeights^\weyl$ as an algebra and thus we may fix
$\hilbertideal{m} = (\augmentationideal{m} \cap \CWeights^\weyl)$ 
as the multiplicative Hilbert ideal in $\CWeights$.

\begin{proposition}\label{prop_GradedCoinvariantIso}
The ideal $\hilbertideal{m}^*:=(h^*\,\vert\,h\in\hilbertideal{m})$ in $\GradedWeights{m}$ is generated by the set $\{(\orb{1}-1)^*,\ldots,(\orb{n}-1)^*\}$. 
Equip $\GradedWeights{m} / \hilbertideal{m}^*$ with the induced multiplicative action. 
Then the map 
\[
    \CWeights_\weyl \to \GradedWeights{m} / \hilbertideal{m}^* ,\,
    f + \hilbertideal{m} \mapsto f^* + \hilbertideal{m}^*
\]
is a well-defined $\weyl$-module isomorphism. 
\end{proposition}
\begin{proof}
To see that the map is well-defined, let $f-g\in\hilbertideal{m}$, that is, $f-g=\sum_i k_i\,(\orb{i}-1)$ for some $k_i\in\CWeights$. 
Then $(f-g)^*=\sum_i k_i^*\,(\orb{i}-1)^* \in \hilbertideal{m}^*$. 
Furthermore, since every element $\tilde{f}\in\GradedWeights{m}$ is of the form $\tilde{f}=[f_0]+[f_1]+[f_2]+\ldots$ for some $f_d\in\augmentationideal{m}^d$,
the inverse is given by $\tilde{f}+\hilbertideal{m}^* \mapsto f_0 + f_1 + f_2 + \ldots + \hilbertideal{m}$ and thus the map is bijective. 
Finally, equivariance follows from \Cref{prop_GradedEquiv}. 
\end{proof}

\begin{theorem}\label{thm_CoinvariantTranserAdditiveMultiplicative}
Let $\{H_\ell\,\vert\,1\leq \ell \leq \nops{\weyl}\} \subseteq \SymWeights$ be a 
normal set for $\SymWeights_\weyl$ with $H_\ell$ homogeneous of degree $d_\ell$. Then 
\[
    \{\phi^{(d_\ell)}(H_\ell)\,\vert\,1\leq \ell \leq \nops{\weyl}\} 
\]
is a normal set for $\GradedWeights{m} / \hilbertideal{m}^*$. 
\end{theorem}
\begin{proof}
Thanks to \Cref{thm:Regual_representation_multiplicative}, 
we only need to show that the $\phi^{(d_\ell)}(H_\ell) \in \GradedWeights{m}$ 
are linearly independent modulo $\hilbertideal{m}^*$. 
We may assume without loss of generality that 
\[
    d_1=\ldots=d_{k_1} < d_{k_1+1} = \ldots = d_{k_2} < d_{k_2 + 1} = \ldots = d_{k_3} < \ldots
\]
with 
$k_0=0$ and $k_r = \nops{\weyl}$. 
Let $\lambda_\ell\in \C$ be scalars such that 
\begin{align*}
    \hilbertideal{m}^*
\ni&\sum\limits_{\ell=1}^{\nops{\weyl}} \lambda_\ell\,\phi^{(d_\ell)}(H_\ell)
=   \sum\limits_{j=1}^r \sum\limits_{\ell=k_{j-1}+1}^{k_j} \lambda_\ell \, \phi^{(d_{k_j})}(H_\ell) \\
=&  \sum\limits_{j=1}^r \phi^{(d_{k_j})}\left( \sum\limits_{\ell=k_{j-1}+1}^{k_j} \lambda_\ell \, H_\ell \right)
\in \bigoplus\limits_{j=1}^r \augmentationideal{m}^{d_{k_j}} / \augmentationideal{m}^{d_{k_j}+1} .
\end{align*}
This decomposition implies
\[
    \phi^{(d_{k_j})}\left( \sum\limits_{\ell=k_{j-1}+1}^{k_j} \lambda_\ell \, H_\ell \right)
\in \hilbertideal{m}^* \cap \augmentationideal{m}^{d_{k_j}} / \augmentationideal{m}^{d_{k_j}+1}
\]
and thus, using the $\weyl$-equivariance from \Cref{prop_GradedEquiv}, 
\[
    \sum\limits_{\ell=k_{j-1}+1}^{k_j} \lambda_\ell \, H_\ell
\in \hilbertideal{a} \cap \SymWeights_{d_{k_j}}. 
\]
In particular, since the $H_\ell$ form a normal set for $\SymWeights/\hilbertideal{a}$,
their only linear combination in $\hilbertideal{a}$ is $0$, which finally yields $\lambda_\ell=0$. 
\end{proof}

\begin{algorithm}
\caption{Coinvariant Transfer}\label{alg_CoinvariantBasisTransfer}
\KwData{$\coinvariantbasis{a} \subseteq \C[X_1,\ldots,X_n]$ 
a normal set for $\SymWeights_\weyl$}
\KwResult{$\coinvariantbasis{m} \subseteq \C[x_1^{\pm 1},\ldots,x_n^{\pm 1}]$ 
a normal set for $\CWeights_\weyl$}
$K = \emptyset$\;
$L = \coinvariantbasis{a}$\;
$\augmentationideal{m}=(1-x_{i}^{-1}\,\vert\,1\leq i\leq n)$\;
\While{$L \neq \emptyset$}{
    \textbf{choose} $H\in L$\;
    $d=\deg(H)$\;
    $h=H\vert_{X_i\mapsto 1-x_{i}^{-1}}$\;
    $h=\mathrm{NF}(h,\augmentationideal{m}^{d+1})$ \Comment{normal form}\;
    $K = K \cup \{h\}$\;
    $L = L \setminus \{H\}$\;
}
\textbf{return} $\coinvariantbasis{m}=K$\;
\end{algorithm}

Since the finite dimensional vector spaces $\SymWeights_d$ and $\augmentationideal{m}^d/\augmentationideal{m}^{d+1}$ 
are isomorphic as $\weyl$-modules by the map $\phi^{(d)}$ from \Cref{prop_GradedEquiv}, 
they have coinciding isotypic decompositions
\begin{equation}
        \SymWeights_d
\cong   \augmentationideal{m}^d/\augmentationideal{m}^{d+1}
\cong   \bigoplus\limits_{i=1}^{\small \mbox{\#irrep.}} m_i^{(d)} \cdot W_i ,
\end{equation}
where $\mbox{\#irrep.}$ is the number of irreducible $\weyl$-modules $W_i$, which appear with multiplicity $m_i^{(d)}$ \cite{serre77}. 
A vector space basis for (a subspace of) 
\begin{equation}
    \SymWeights = \bigoplus\limits_{d\in\N} \SymWeights_d
    \quad \mbox{or} \quad
    \GradedWeights{m} = \bigoplus\limits_{d\in\N} \augmentationideal{m}^d/\augmentationideal{m}^{d+1}
\end{equation}
respecting the isotypic decomposition of the homogeneous components is called \textbf{symmetry adapted} 
(see \cite{DerksenKemper,gatermann2004symmetry} for general finite groups). 

\begin{corollary}
Assume that the normal set $\{H_\ell\,\vert\,1\leq \ell \leq \nops{\weyl}\}$ for $\SymWeights_\weyl$
in \Cref{thm_CoinvariantTranserAdditiveMultiplicative} is symmetry adapted. 
Then the same holds for the normal set $\{\phi^{(d_\ell)}(H_\ell)\,\vert\,1\leq \ell \leq \nops{\weyl}\}$ 
for $\GradedWeights{m}/\hilbertideal{m}^*$. 
\end{corollary}

With \Cref{prop_GradedCoinvariantIso}, we have
\begin{equation}
        \CWeights 
\cong   \CWeights^\weyl \otimes_\C \CWeights_\weyl
\cong   \CWeights^\weyl \otimes_\C \GradedWeights{m} / \hilbertideal{m}^* 
\end{equation}
and can therefore derive a symmetry adapted basis for $\CWeights$ by multiplication with invariants in $\hilbertideal{m}$ 
(see \cite{debus_riener_2023} for additive invariants). 

\begin{example}\label{example_coinvariant_transfer_C2}
Recall from  {\Cref{example_c2_fundamental_invariants}}, 
that the weight lattice $\Weights=\Z^2$ is generated by 
the fundamental weights $\fweight{1},\fweight{2}$ defining variables $X_i,x_i$ 
and by the standard basis $e_1,e_2$ defining variables $Y_i,y_i$ 
for $\SymWeights$ and $\CWeights$ respectively. 

For a symmetry adapted basis of the additive coinvariant space 
$\SymWeights_\weyl$ of $\RootC[2]$ with coordinates $Y_1,Y_2$, 
we choose higher Specht polynomials \cite{morita1998higher}. 
From this construction, 
we display in {\Cref{table:C_2 additive coinvariants}} symmetry adapted bases 
for both the additive and multiplicative coinvariant spaces. 
Although they are not a monomial basis such as the normal set in {\Cref{example:C_2 normal set}}, 
we can track the isotypic decomposition. 

\begin{center}
\begin{table}[h!] 
    \centering
    \renewcommand{\arraystretch}{1.5}
		\scalebox{0.72}{
    \begin{tabular}{|c|c|c|}
        \hline
        \textbf{Irrep} & \textbf{Dim} & \textbf{Basis Elements}  \\
        \hline
        $ \mathbb{S}^{((2),(0))} $ & 1 & $ 1 $ \\
        \hline
        $ \mathbb{S}^{((1),(1))} $ & 2 & $ Y_1, Y_2 $ \\
                        &   &   $1 - 1/y_1, 1/y_1 - 1/(y_1\,y_2) $ \\
                        &   &   $X_1, -X_1 + 2\,X_2 $ \\
                        &   &   $1 - 1/x_1, 1/x_1 - 1/x_2 $ \\
                        &   & $ Y_2^2 \cdot Y_1, \, Y_1^2 \cdot Y_2 $ \\
                        &   &   $3/y_2 - 1/y_1^2 + 1/(y_1^2\,y_2) + y_1\,y_2 + 4/y_1 - 3/(y_1\,y_2) - 1, -1/y_2 + 1/y_1^2 - 1/(y_1^2\,y_2) - 2/y_1 + 2/(y_1\,y_2) + 1 $ \\
                        &   &   $X_1^3 - 4\,X_1^2\,X_2 + 4\,X_1\,X_2^2, -X_1^3 + 2\,X_1^2\,X_2 $ \\
                        &   &   $3\,x_1/x_2 - 1/x_1^2 + 1/(x_2\,x_1) + x_2 + 4/x_1 - 3/x_2 - 1, -x_1/x_2 + 1/x_1^2 - 1/(x_2\,x_1) - 2/x_1 + 2/x_2 + 1 $ \\
        \hline
        $ \mathbb{S}^{((0),(2))} $ & 1 & $ Y_1 \cdot Y_2 $ \\
                        &   &   $1/(y_1^2\,y_2) - y_1 - 2/y_1 - 1/(y_1\,y_2) + 3 $ \\
                        &   &   $-X_1^2 + 2\,X_1\,X_2 $ \\
                        &   &   $1/(x_2\,x_1) - x_1 - 2/x_1 - 1/x_2 + 3 $ \\
        \hline
        $ \mathbb{S}^{((1,1),(0))} $ & 1 & $ Y_1^2 - Y_2^2 $ \\
                        &   &   $2/(y_1^2\,y_2) - y_1\,y_2 - 2/y_1 - 3/(y_1\,y_2) + 4 $ \\
                        &   &   $4\,X_1\,X_2 - 4\,X_2^2 $ \\
                        &   &   $2/(x_2\,x_1) - x_2 - 2/x_1 - 3/x_2 + 4 $ \\
        \hline
        $ \mathbb{S}^{((0),(1,1))} $ & 1 & $ (Y_1^2 - Y_2^2) \cdot Y_1 \cdot Y_2 $ \\
                        &   &   $6/y_2 + 4/y_1^2 + 14/(y_1^2\,y_2) - 10\,y_1 - 10\,y_1\,y_2 - 18/y_1 - 10/(y_1\,y_2) + 24 $ \\
                        &   &   $-4\,X_1^3\,X_2 + 12\,X_1^2\,X_2^2 - 8\,X_1\,X_2^3 $ \\
                        &   &   $6\,x_1/x_2 + 4/x_1^2 + 14/(x_2\,x_1) - 10\,x_1 - 10\,x_2 - 18/x_1 - 10/x_2 + 24 $ \\
        \hline
    \end{tabular} }
    \caption{
    %Additive and multiplicative coinvariants for the Weyl group $\weyl\cong\mathfrak{S}_2 \wr \{\pm 1\}$ of $\RootC[2]$ in the standard basis ($Y_1=e_1,Y_2=e_2$) and in the basis of fundamental weights ($X_1=\fweight{1}=e_1,X_2=\fweight{2}=(e_1+e_2)/2$). 
    In case of $\RootC[2]$, applying \Cref{alg_CoinvariantBasisTransfer} 
    to a symmetry adapted basis for $\SymWeights_\weyl$ (consisting of higher Specht polynomials in $Y_1=e_1,Y_2=e_2$)
    yields a symmetry adapted basis for $\CWeights_\weyl$ (in $y_1,y_2$). 
    The change of coordinates to $X_1=\fweight{1}=e_1,X_2=\fweight{2}=(e_1+e_2)/2$ 
    yields the basis in $x_1,x_2$. 
    }
    \label{table:C_2 additive coinvariants}
\end{table}
\end{center}

\end{example}

%%%%%%%%

\section{Homomorphisms for types $\RootA$ and $\RootC$}\label{sec:homomorphisms}

Although $\CWeights$ and $\SymWeights$ are not isomorphic as rings, 
one can find homomorphisms of $\weyl$-modules for types $\RootA$ and $\RootC$. 

The reason why we find the linear maps only for these types is that the change of basis from 
fundamental weights $x_i$ to standard $y_i$ over the Laurent polynomials is bijective only for types $\RootA$ and $\RootC$ 
among the infinite families of Weyl groups, see \cite[Lem.~5.2,~6.2]{metzlaff2024}. 

\subsection{Type $\RootA$}

The symmetric group $\mathfrak{S}_n$ acts on the affine algebraic varieties
\begin{align*}
                    &   V:=\{Y=(Y_1,\ldots,Y_n)\in\C^n\,\vert\, Y_1 + \ldots + Y_n = 0\} \\
\mbox{and} \quad    &   W:=\{y=(y_1,\ldots,y_n)\in\C^n\,\vert\, y_1 \cdots y_n - 1 = 0\}
\end{align*} 
by permutation of coordinates. 
This defines a representation $\rho:\,\mathfrak{S}_n \to \mathrm{GL}_n(\C)$ with 
induced actions on the coordinate rings 
$\C[V] \cong \C[Y_1,\ldots,Y_n]/(Y_1+\ldots+Y_n)$ and 
$\C[W] \cong \C[y_1,\ldots,y_n]/(y_1\cdots y_n - 1)$. 

Recall from \cite[Pl. I]{bourbaki456} that $\weyl = \mathfrak{S}_n$ is the Weyl group of the root system $\RootA$. 
Denote by $\fweight{1},\ldots,\fweight{n-1}$ the fundamental weights of the weight lattice $\Weights$. 
This defines a representation $\pi:\,\mathfrak{S}_n \to \mathrm{GL}_{n-1}(\Z)$ 
with an additive action on $\SymWeights \cong \C[X_1,\ldots,X_{n-1}]$ 
and a multiplicative action on $\CWeights \cong \C[x_{1}^{\pm 1},\ldots,x_{n-1}^{\pm 1}]$. 
Furthermore, we have
\begin{equation}
            \weyl(\fweight{1}) 
=           \{\fweight{1},\fweight{2}-\fweight{1},\ldots,\fweight{n-1}-\fweight{n-2},-\fweight{n-1}\} ,
\end{equation}
which leads to the following statement (see also \cite[\S 5]{metzlaff2024}). 

\begin{lemma}\label{lem: equiv_A}
The maps 
\[
  \C[V] \to \SymWeights,\, 
  Y_1\mapsto X_1,\,
  Y_2\mapsto X_2 - X_{1},\,
  \ldots,\,
  Y_{n-1}\mapsto X_{n-1} - X_{n-2},\,
  Y_n\mapsto -X_{n-1}
\]
and
\[
  \C[W] \to \CWeights,\,
  y_1\mapsto x_1,\,
  y_2\mapsto x_2\,x_{1}^{-1},\,
  \ldots,\,
  y_{n-1}\mapsto x_{n-1}\,x_{n-2}^{-1},\,
  y_n\mapsto x_{n-1}^{-1} 
\]
are well-defined $\rho$-$\pi$-equivariant algebra isomorphisms. 
\end{lemma}
We note that the isomorphism $\C[W] \to \CWeights$ was already observed in \cite[\S 9]{garsia1984}.
\begin{proof}

The choice of coordinates 
$X_i=\fweight{i}, Y_i=\fweight{i}-\fweight{i-1},x_i=\mathfrak{e}^{\fweight{i}},y_i=\mathfrak{e}^{\fweight{i}-\fweight{i-1}}$, 
by which the $Y_i$ and $y_i$ form $\weyl$-orbits of length $n$, 
implies that the maps are well-defined and $\rho$-$\pi$-equivariant. 
For bijectiveness, note that the inverse maps are given by  $X_i\mapsto Y_1+\ldots+Y_i$ and $x_i \mapsto y_1\cdots y_i$. 
\end{proof}

Hence, the induced actions on $\C[V]$ and $\C[W]$ can be identified with the additive and multiplicative action respectively. 

\begin{example}
The Weyl group $\weyl$ of $\RootA[2]$ is generated by $s_1,s_2$ with
\[
    \rho(s_1) = \begin{pmatrix} 0&1&0\\1&0&0\\0&0&1 \end{pmatrix},\,
    \rho(s_2) = \begin{pmatrix} 1&0&0\\0&0&1\\0&1&0 \end{pmatrix} 
\]
and
\[
    \pi (s_1) = \begin{pmatrix} -1&0\\1&1 \end{pmatrix},\,
    \pi (s_2) = \begin{pmatrix} 1&1\\0&-1 \end{pmatrix} .
\]
\end{example}

We now compute a coinvariant basis for $\CWeights_\weyl$. 
The fundamental theorem of symmetric polynomials states that the invariant ring 
$\C[y_1,\ldots,y_n]^{\rho(\mathfrak{S}_n)}$ 
is, as an algebra, generated by the elementary symmetric polynomials
\begin{equation}
\sigma_i(y) = \sum_{\substack{J \subseteq \{1 ,\ldots, n\} \\ \mbox{with } \vert J \vert = i }} \, \prod_{j\in J} y_j.
\end{equation}
In particular, $\C[W]^\weyl \cong \C[\sigma_1(y),\ldots,\sigma_{n-1}(y)]/(\sigma_n(y)-1)$. 

On the other hand, the multiplicative invariant ring 
$\CWeights^\weyl \cong \C[x_{1}^{\pm 1},\ldots,x_{n-1}^{\pm 1}]^{\pi(\mathfrak{S}_n)}$ 
is generated by orbit polynomials $\orb{i}(x)$. 
Recall from \Cref{lem: equiv_A} that
\[
    \Psi_\mathrm{A}: \C[W]\to\CWeights, \, y_1\mapsto x_1,y_2\mapsto x_2\,x_{1}^{-1},\ldots,y_{n-1}\mapsto x_{n-1}\,x_{n-2}^{-1},y_n\mapsto x_{n-1}^{-1} \label{eq:map1} 
\] 
is an equivariant algebra isomorphism. 

\begin{proposition}
\label{cor:A_coinvariant_algebra_isomorphism}
The map 
\[
 \C[y]/(\sigma_i(y)\,\vert\,1\leq i\leq n) \to \C[x^\pm]/(\orb{i}\,\vert\,1\leq i\leq n-1)   
\]
that takes $h(y)$ to $\Psi_\mathrm{A} (h(y))$ is a $\weyl$-module isomorphism. 
\end{proposition}
\begin{proof}
By \cite[Prop.~5.1]{metzlaff2024}, $\Psi_\mathrm{A}(\sigma_i(y))$ is equal to $\theta_i(x)$ up to a scalar and thus
\begin{align*}
        \C[W]_\weyl 
&\cong   \C[y]/(\sigma_1(y),\ldots,\sigma_{n-1}(y),\sigma_n(y)-1) \\
&\cong   \C[x^\pm]/(\orb{1}(x),\ldots,\orb{n-1}(x)) 
\cong   \CWeights_\weyl. 
\end{align*}
Then the statement follows as in \Cref{lem: quotient size}. 
\end{proof}

The irreducible representations of $\mathfrak{S}_n$ are called \textbf{Specht modules} and are denoted by $\mathbb{S}^\lambda$ where $\lambda$ is a partition of $n$ (we refer to \cite{Sagan} for details).  

\begin{example} \label{example:A_2 additive and multiplicative basis}
The Weyl group $\weyl=\mathfrak{S}_3$ of $A_2$ has three irreducible representations $\mathbb{S}^\lambda$,
where $\lambda$ is a partition of $3$. 
Moreover the group algebra $\C[\weyl]$ has isotypic decomposition $1 \cdot \mathbb{S}^{(3)} \oplus 2 \cdot \mathbb{S}^{(2,1)} \oplus 1 \cdot \mathbb{S}^{(1,1,1)}. $ 
In \Cref{table:A_2 additive and multiplicative coinvariants}, 
we start with a symmetry adapted basis of the additive coinvariant space. 
The additive basis consists of higher Specht polynomials and was calculated with the M2 package SpechtModule  {\cite{SpechtModule}}. 
 
\begin{table}[h!]
    \centering
    \begin{tabular}{|c|c|c|c|}
        \hline
        & $\mathbb{S}^{(3)}$ & $\mathbb{S}^{(2,1)}$ & $\mathbb{S}^{(1,1,1)}$ \\
        \hline
        a & $1$ & {\tiny $\{Y_1-Y_2, \, Y_1-Y_3\}$} & {\tiny $(Y_1-Y_2)(Y_1-Y_3)(Y_2-Y_3)$} \\
             &     & {\tiny $\{(Y_1-Y_2)Y_3, \, (Y_1-Y_3)Y_2\}$} & \\
        \hline
        m & $1$ & {\tiny $\{f_{1,1} := x_1-x_2x_1^{-1}, \, f_{1,2} := x_1-x_2^{-1}\}$} & {\tiny $g:= x_1x_2-x_1^{-1}x_2^2+x_1^{-2}x_2$} \\
              &     & {\tiny $\{f_{2,1}:=x_1x_2^{-1}-x_1^{-1}, \, f_{2,2}:=x_2-x_1^{-1}\}$} & {\tiny $-x_1^2x_2^{-1}+x_1x_2^{-2}-x_1^{-1}x_2^{-1}$} \\
        \hline
    \end{tabular}
    \caption{
    In case of $\RootA[2]$, 
    the image of a symmetry adapted basis for $\SymWeights_\weyl$ 
    under the map $\Psi_\mathrm{A}$ from \Cref{cor:A_coinvariant_algebra_isomorphism} 
    is a symmetry adapted basis for $\CWeights_\weyl$.
    }
    \label{table:A_2 additive and multiplicative coinvariants}
\end{table}
\end{example}

A filtration of $\CWeights$ based on the Vorono\"i cell of the dual lattice 
was introduced in \cite{chromatic22,metzlaff2024onSymmetryAdaptedbases} 
with the purpose of optimization of trigonometric polynomials with crystallographic symmetry. 
In said article, a multiplicity pattern of the irreducible representations was computationally observed for $\RootA[2]$. 
The symmetry adapted basis from \Cref{example:A_2 additive and multiplicative basis} allows us to prove this observation. 

We equip the real vector space $\cartan = \spn_\R\Weights$ with an inner product
$\sprod{\cdot,\cdot}: \cartan \times \cartan \to \R$, 
so that $\weyl$ is orthogonal, that is, $\sprod{s(X),Y}=\sprod{X,s^{-1}(Y)}$ for all $X,Y\in\cartan$. 
The \textbf{coroot lattice} $\Corootlattice$ is the dual lattice of $\Weights$, that is, 
$\Corootlattice = \{\lambda \in \cartan\,\vert\,\forall \weight\in\Weights:\,\sprod{\lambda,\weight} \in \Z\}$. 

Let $\norm{\cdot} := \sqrt{\sprod{\cdot,\cdot}}$ be the induced norm. 
Then the \textbf{Vorono\"i cell} of $\Corootlattice$ is 
$\Vor(\Corootlattice) := \{X \in \cartan\,\vert\,\forall\,\lambda\in\Corootlattice:\,\norm{X} \leq \norm{\lambda - X} \}$. 

Since $\Weights$ is full-dimensional, so is $\Corootlattice$ and in particular $\Vor(\Corootlattice)$ is compact. 
Thus, for $d \in \N$, the subset $\Weights_d:= \Weights \cap (d\cdot \Vor(\Corootlattice))$ of the lattice is finite, 
where $d\cdot \Vor(\Corootlattice)$ denotes the dilation of $\Vor(\Corootlattice)$ by $d$, see \Cref{fig_voronoi_dilation}. 
We define $\CWeights_d$ as the finite-dimensional vector subspace of $\CWeights$ 
containing all Laurent polynomials with support in $\Weights_d$. 

By \cite[Ch. VI, \S 2, Prop. 4]{bourbaki456} and \cite[Ch. 21, \S 3, Thm. 5]{conway1988a}, 
we have $\Vor(\Corootlattice) = \weyl\fundom$, 
where $\fundom$ is the closure of an alcove of the affine Weyl group $\weyl \ltimes \Corootlattice$ so that $0\in\fundom$. 
In particular, the set $\weyl (\Weights \cap (d \cdot \fundom)) = \Weights \cap (d\cdot \Vor(\Corootlattice)) = \Weights_d$ is $\weyl$-stable
and thus $\weyl$ acts multiplicatively on $\CWeights_d$. 

\begin{figure}
\centering
\begin{tikzpicture}[scale=0.60]
\def\hexsize{1} 
\def\layers{4}  

\pgfmathsetmacro{\graystep}{80 / (\layers + 1)} 
\foreach \layer in {1,...,\layers} {
    \pgfmathsetmacro{\graylevel}{100 - \layer * \graystep}
    \fill[gray!\graylevel] (0:\layer * \hexsize)
    \foreach \i in {1,...,6} {
        -- (\i * 60:\layer * \hexsize)
    } -- cycle;
}

\foreach \layer in {0,...,\layers} {
    \ifnum\layer>0
        \draw[gray, thick] (0:\layer * \hexsize)
        \foreach \i in {1,...,6} {
            -- (\i * 60:\layer * \hexsize)
        } -- cycle;
    \fi
}

\foreach \layer in {0,...,\layers} {
    \foreach \i in {0,1,...,5} {
        \pgfmathsetmacro{\angle}{\i * 60}
        \foreach \j in {0,...,\layer} {
            \pgfmathsetmacro{\x}{(\layer - \j) * \hexsize * cos(\angle) + \j * \hexsize * cos(\angle + 60)}
            \pgfmathsetmacro{\y}{(\layer - \j) * \hexsize * sin(\angle) + \j * \hexsize * sin(\angle + 60)}
            \node[circle, fill=blue, inner sep=2pt] at (\x, \y) {};
        }
    }
}

\node[circle, fill=blue, inner sep=2pt] at (0, 0) {};

\end{tikzpicture}
\caption{
The subsets $\Weights_0 \subset \Weights_1 \subset \Weights_2 \subset \Weights_3 \subset \Weights_4$ 
of the hexagonal lattice in the plane, the weight lattice $\Weights$ of $\RootA[2]$.
}\label{fig_voronoi_dilation}
\end{figure}

\begin{theorem} \label{thm:isotypic decomposition wrt filtration} 
For $A_2$, the isotypic decomposition of $\CWeights_d$ is 
\[ \frac{(d+1)(d+2)}{2} \cdot \mathbb{S}^{(3)} \oplus d(d+1) \cdot \mathbb{S}^{(2,1)} \oplus \frac{(d-1)d}{2} \cdot \mathbb{S}^{(1,1,1)}.\]
\end{theorem}
\begin{proof}
Note that if $\weight \in (d \cdot \Weights) \setminus ((d-1) \cdot \Weights)$, 
then so is $s(\weight)$ whenever $s \in \weyl$, because $\weyl$ is an isometry group. 

Using the symmetry adapted basis of the coinvariant space $\C[x^\pm]_{\weyl}$ from 
\Cref{example:A_2 additive and multiplicative basis}, we find 
\begin{align*}
    \CWeights_0 =& \langle 1 \rangle_\C = \mathbb{S}^{(3)}, \\
    \CWeights_1 =& \CWeights_0 \oplus \langle  \theta_1,\theta_2, 
    f_{1,1},f_{1,2},f_{2,1},f_{2,2} \rangle_\C  
     = 3\cdot\mathbb{S}^{(3)} \oplus 2\cdot \mathbb{S}^{(2,1)} \\
    \CWeights_2 =& \CWeights_1 \oplus  \langle \theta_i^2,\theta_1\theta_2,\theta_i f_{j,k}, g \mid 1 \leq i,j,k \leq 2\rangle_\C    =  6 \cdot \mathbb{S}^{(3)} \oplus 6 \cdot \mathbb{S}^{(2,1)} \oplus 1 \cdot \mathbb{S}^{(1,1,1)}. 
\end{align*}
For integers $a,b\in\N$, we have  
\begin{align*}
    \theta_1^a\theta_2^b \in \CWeights_{a+b}\setminus \CWeights_{a+b-1} , \\
    \theta_1^a\theta_2^b f_{j,k} \in \CWeights_{a+b+1} \setminus \CWeights_{a+b} , \\
    \theta_1^a\theta_2^b g \in \CWeights_{a+b+2} \setminus \CWeights_{a+b+1} ,
\end{align*} 
which can be seen from $\Weights_d + \Weights_{d'} \subset \Weights_{d+d'}$, 
see \cite[Lem.~2.1~(2)]{metzlaff2024onSymmetryAdaptedbases}, 
and by observing that $1$ is the coefficient of the four monomials 
\[
x_1^a\,x_2^b,\,x_1^{a+1}\,x_2^b,\,x_1^{a}\,x_2^{b+1},\,x_1^{a+1}\,x_2^{b+1}
\]
in the Laurent polynomials 
\[ 
\theta_1^a\,\theta_2^b,\,\theta_1^a\,\theta_2^b\,f_{1,1},\,\theta_1^a\,\theta_2^b\,f_{2,2},\,\theta_1^a\,\theta_2^b\,g . 
\]
Using the decomposition $\CWeights = \C[\theta_1,\theta_2] \otimes_\C \CWeights/(\theta_1,\theta_2)$, 
we derive the claimed multiplicities of $\mathbb{S}^{\lambda}$ in $\CWeights_d$, because 
\[ 
\nops{\{(a,b)\in \N^2 \mid a+b\leq d\} } =\sum\limits_{i=0}^d (d+1-i) = \frac{(d+1)(d+2)}{2} .
\]
\end{proof}

Finding the analogous isotypic decompositions for $\RootA$ seems very challenging. 
We have $\dim_{\C} \CWeights_1=2^{n}-1$ and, for $n=4$, we found 
\begin{equation}
\CWeights_1 = 4 \cdot \mathbb{S}^{(4)} \oplus 3 \cdot \mathbb{S}^{(3,1)} \oplus 1 \cdot \mathbb{S}^{(2,2)}. 
\end{equation}
We do not foresee a pattern for the multiplicities of the isotypic decomposition of $\CWeights_1$ for larger $n$. 

\subsection{Type $\RootC$}

The hyperoctahedral group $\mathfrak{S}_n \wr \{\pm 1\}$ acts on the affine algebraic varieties
\begin{equation}
    V
:=  \{Y=(Y_1,\ldots,Y_n)\in\C^n\} 
    \quad \mbox{and} \quad
    W
:=  \{y=(y_1,\ldots,y_n)\in(\C\setminus\{0\})^n\}
\end{equation} 
by permutation as well as sign-change and inversion of coordinates respectively. 
This defines a representation $\rho:\,\mathfrak{S}_n \wr \{\pm 1\} \to \mathrm{GL}_n(\C)$ with 
induced actions on the coordinate rings 
$\C[V] \cong \C[Y_1,\ldots,Y_n]$ and 
$\C[W] \cong \C[y_1^{\pm 1},\ldots,y_n^{\pm 1}]$. 

Recall from \cite[Pl. I]{bourbaki456} that $\weyl = \mathfrak{S}_n \wr \{\pm 1\}$ is the Weyl group of the root system $\RootC$. 
Denote by $\fweight{1},\ldots,\fweight{n}$ the fundamental weights of the weight lattice $\Weights$. 
This defines a representation $\pi:\,\mathfrak{S}_n \wr \{\pm 1\}\to \mathrm{GL}_{n}(\Z)$ 
with an additive action on the symmetric algebra $\SymWeights \cong \C[X_1,\ldots,X_n]$ 
and a multiplicative action on the group algebra $\CWeights \cong \C[x_{1}^{\pm 1},\ldots,x_{n}^{\pm 1}]$. 
Furthermore, we have
\begin{equation}
            \weyl(\fweight{1}) 
=           \{\fweight{1},\fweight{2}-\fweight{1},\ldots,\fweight{n}-\fweight{n-1}\} ,
\end{equation}
which leads to the following statement analogous to type $\RootA$ 
(see also \cite[\S 6]{metzlaff2024}). 

\begin{lemma}\label{lem: equiv_C}
The maps 
\[
    \C[V] \to \SymWeights,\,
    Y_1\mapsto X_1,\,
    Y_2\mapsto X_2-X_{1},\,
    \ldots,\,
    Y_{n}\mapsto X_{n} - X_{n-1}
\]
and
\[
    \C[W] \to \CWeights,\,
    y_1\mapsto x_1,\,
    y_2\mapsto x_2\,x_{1}^{-1},\,
    \ldots,\,
    y_{n}\mapsto x_{n}\, x_{n-1}^{-1}
\]
are well-defined $\rho$-$\pi$-equivariant algebra isomorphisms. 
\end{lemma}

Hence, the induced actions on $\C[V]$ and $\C[W]$ can be identified with the additive and multiplicative action respectively. 

\begin{example}
The Weyl group $\weyl$ of $\RootC[2]$ is generated by $s_1,s_2$ with
\[
    \rho(s_1) = \begin{pmatrix} 0&1\\1&0 \end{pmatrix},\,
    \rho(s_2) = \begin{pmatrix} 1&0\\0&-1 \end{pmatrix} 
\]
and
\[
    \pi (s_1) = \begin{pmatrix} -1&0\\1&1 \end{pmatrix},\,
    \pi (s_2) = \begin{pmatrix} 1&2\\0&-1 \end{pmatrix} .
\]
\end{example}

As in \Cref{lem: equiv_A}, 
we find a $\weyl$-module isomorphism between $\C[V]$ and $\C[W]$. 

\begin{proposition}\label{lem:C_isomorphism_symmetric_algebra_and_laurent_ring}
The map $\Psi_\mathrm{C} : \C[V] \to \C[W]$ with
\begin{align*}
Y_i^{2k} \mapsto y_i^k+y_i^{-k}, Y_i^{2k-1} \mapsto y_i^k-y_i^{-k}, \text{ for all } i \in \{1,\ldots,n\}, \,k \in \N,
\end{align*}
and $\prod_{i=1}^nY_i^{\alpha_i} \mapsto \prod_{i=1}^n\Psi_\mathrm{C}(Y_i^{\alpha_i})$ is a $\rho$-$\pi$-equivariant isomorphism. 
Moreover, the image of $\hilbertideal{a}$ is not a proper ideal in $\C[W]$. 
\end{proposition}
\begin{proof}
By construction the map is well-defined, injective and $\rho$-$\pi$-equivariant.
The map is also surjective since 
\[
    y_i^{\pm k} 
=   \frac{y_i^k+y_i^{-k}}{2} \pm \frac{y_i^k-y_i^{-k}}{2} = \frac{\Psi_\mathrm{C}(Y_i^{2k})\pm \Psi_\mathrm{C}(Y_i^{2k-1})}{2} 
\in \Ime (\Psi_\mathrm{C}). 
\]
Assume that $\hilbertideal{a}$ is mapped to an ideal $I\subseteq \C[W]$, that is, 
\[
    \Psi_\mathrm{C}\left(\sum\limits_{i=1}^n Y_i^2\right), 
    \Psi_\mathrm{C}\left(\sum\limits_{i=1}^n Y_i^4\right) \mbox{ and }
    \Psi_\mathrm{C}\left(\sum\limits_{1\leq i < j \leq n} Y_i^2\,Y_j^2\right) 
\in I.
\] 
Since $I$ is closed under multiplication and addition, this would imply
\[
    2\,n
=   \left(\Psi_\mathrm{C}\left(\sum\limits_{i=1}^n Y_i^2\right)\right)^2
-   \Psi_\mathrm{C}\left(\sum\limits_{i=1}^n Y_i^4\right)
-   2\,\Psi_\mathrm{C}\left(\sum\limits_{1\leq i < j \leq n} Y_i^2\,Y_j^2\right)
\in I,
\]
that is, $I = \C[W]$. 
\end{proof}

In particular, the image of the additive Hilbert ideal is not the multiplicative Hilbert ideal.
Thus the map $\Psi_\mathrm{C}$ does not extend to a $\weyl$-homomorphism of the coinvariant spaces.

\begin{example} \label{example:C_2 additive coinvariant basis goes to multiplicative one}
The additive coinvariants in \Cref{table:C_2 additive coinvariants} are mapped to 
linearly independent equivariants in $\C[W]$ by \Cref{lem:C_isomorphism_symmetric_algebra_and_laurent_ring}.
The irreducible representations of $\mathfrak{S}_n \wr \{\pm 1\}$ are also called Specht modules 
and are indexed by pairs of partitions $(\lambda,\mu)$ where the sizes of $\lambda$ and $\mu$ sum to $n$ 
(see for example  {\cite{musili1993representations}}). 
\begin{table}[h!] 
    \centering
    \renewcommand{\arraystretch}{1.5}
    \begin{tabular}{|c|c|}
        \hline
        \textbf{Irreducible Representation} & \textbf{Multiplicative Equivariants} \\
        \hline
        $ \mathbb{S}^{((2),(0))} $ &  $ 1 $ \\
        \hline
        $ \mathbb{S}^{((1),(1))} $ &  $ y_1-y_1^{-1}, y_2-y_2^{-1} $ \\
                        &    $ (y_2+y_2^{-1})\cdot (y_1-y_1^{-1}) , \, (y_1+y_1^{-1})\cdot (y_2-y_2^{-1}) $ \\
        \hline
        $ \mathbb{S}^{((0),(2))} $ &  $ (y_1-y_1^{-1})\cdot (y_2-y_2^{-1}) $ \\
        \hline
        $ \mathbb{S}^{((1,1),(0))} $ &  $ (y_1+y_1^{-1})- (y_2+y_2^{-1}) $ \\
        \hline
        $ \mathbb{S}^{((0),(1,1))} $ & $ (y_1^2-y_1^{-2})(y_2-y_2^{-1})-(y_1-y_1^{-1})(y_2^2-y_2^{-2}) $ \\
        \hline
    \end{tabular} 
    \caption{
    In case of $\RootC[2]$, 
    the image of a symmetry adapted basis for $\SymWeights_\weyl$ 
    under the map $\Psi_\mathrm{C}$ from \Cref{lem:C_isomorphism_symmetric_algebra_and_laurent_ring} 
    is a symmetry adapted basis for $\CWeights_\weyl$.
    Compare this with the output of \Cref{alg_CoinvariantBasisTransfer} in \Cref{table:C_2 additive coinvariants}. 
    }\label{Table:C_2 image of higher Specht basis}
\end{table}
\end{example}

Since the action of $\weyl$ in the standard basis $y$ is more intuitive than in the action in the basis $x$ of fundamental weights, it might be simpler to relate the additive coinvariant basis consisting of higher Specht polynomials to a symmetry adapted basis for $\C[y^{\pm}]/(\sigma_n(y+y^{-1}))$, where $y+y^{-1}:=(y_1+y_1^{-1},\ldots,y_n+y_n^{-1})$. 
The isomorphisms in \Cref{lem: equiv_C} allow us to consider the multiplicative coinvariant basis in a more intuitive way. 
The invariant ring 
$\C[y_1^{\pm 1},\ldots,y_n^{\pm 1}]^{\rho(\mathfrak{S}_n\wr\{\pm1\})}$ 
is, as an algebra, generated by 
\begin{equation}
    \sigma_i(y+y^{-1}) 
=   \sum_{\substack{J \subseteq \{1 ,\ldots, n\} \\ 
    \mbox{with } \vert J \vert = i }} \, 
    \prod_{j\in J} y_j+y_j^{-1},
\end{equation}
that is, $\C[W]^\weyl \cong \C[\sigma_1(y+y^{-1}),\ldots,\sigma_{n}(y+y^{-1})]$, see \cite[Ex. 3.5.4]{lorenz06}. 

Again, the multiplicative invariant ring 
$\C[x_{1}^{\pm 1},\ldots,x_{n}^{\pm 1}]^{\pi(\mathfrak{S}_n \wr \{\pm 1\})}$ 
is generated by orbit polynomials $\orb{i}$. 
Recall from \Cref{lem: equiv_C} that
\[
    \Gamma : \C[W] \to \CWeights,\,  y_1\mapsto x_1,y_2\mapsto x_2\,x_{1}^{-1},\ldots,y_{n}\mapsto x_{n}\, x_{n-1}^{-1} \label{eq:map2}
\]
is an equivariant algebra isomorphism. 

\begin{proposition}\label{cor:Type C isomorphism of coinvariant spaces}
The map 
\[
\C[y^{\pm}]/(\sigma_i(y+y^{-1})\,\vert\,1\leq i\leq n)  \to \C[x^\pm]/(\orb{i}(x)\,\vert\,1\leq i\leq n)
\]
that takes $h(y)$ to $\Gamma (h(y))$ is a $\weyl$-module isomorphism.
\end{proposition}
\begin{proof}
Since $\Gamma (\sigma_i(y+y^{-1}))=2^i{n \choose i}\theta_i(x)$ by \cite[Prop.~6.1]{metzlaff2024}, this follows analogously to \Cref{cor:A_coinvariant_algebra_isomorphism}. 
\end{proof}

%%%%%%%%%%%

\section{Conclusion and Open Questions}

Given a Weyl group $\weyl$ of a crystallographic root system in a real vector space with weight lattice $\Weights$, 
we show that the multiplicative coinvariant space $\CWeights_\weyl$ affords the regular representation of $\weyl$, 
introduce a space of multiplicative $\weyl$-harmonics that also affords the regular representation, 
and draw analogies to the additive coinvariant space $\SymWeights_\weyl$. 

Using the technique of graded transfer, 
an algorithm is designed to transform 
an additive coinvariant basis 
into a multiplicative one. 
Therefore, the two coinvariant spaces and in particular the algebras 
$\CWeights$ and $\SymWeights$ are isomorphic $\weyl$-modules. 

Computationally, the construction of explicit $\weyl$-homomorphisms proved challenging. 
Such homomorphisms were only found for types $\RootA$ (\Cref{lem: equiv_A}) and partly $\RootC$ (\Cref{lem: equiv_C}). 

\begin{question}\label{question1}
Are there explicit $\weyl$-homomorphism between 
$\CWeights$ and $\SymWeights$ for the types
$\RootB,\,\RootD,\,\RootE[6],\,\RootE[7],\,\RootE[8]$ or $\RootF[4]$?\end{question}

For type $\RootC[2]$, 
higher Specht polynomials form a symmetry adapted basis for $\SymWeights_\weyl$. 
We have two strategies to construct one for $\CWeights_\weyl$. 
First, using \Cref{alg_CoinvariantBasisTransfer}, 
we compute a symmetry adapted basis for 
$\CWeights_\weyl$ in \Cref{table:C_2 additive coinvariants}. 
Second, applying \Cref{lem:C_isomorphism_symmetric_algebra_and_laurent_ring} 
yields an equivariant isomorphism between $\SymWeights$ and $\CWeights$, 
which does NOT induce an isomorphism of the coinvariant spaces. 
Nonetheless, the resulting polynomials in \Cref{Table:C_2 image of higher Specht basis} 
form a symmetry adapted basis for $\CWeights_\weyl$, 
which differs from the one in \Cref{table:C_2 additive coinvariants}. 

\begin{question}
For general type $\RootC$,
is the image of a higher Specht polynomial basis for $\SymWeights_\weyl$ 
under the map $\Psi_\mathrm{C}$ 
always a (symmetry adapted) basis 
for the multiplicative coinvariant space $\CWeights_\weyl$?
\end{question}

Expanding on type $\RootC$, 
the multiplicative coinvariant space is 
$\CWeights_\weyl\cong\C[y^{\pm}]/(\sigma_i(y+y^{-1}))$ 
(see \Cref{cor:Type C isomorphism of coinvariant spaces}) 
on which the group $\weyl\cong\mathfrak{S}_n \wr \{\pm 1\}$ acts intuitively 
by permutation and inversion of variables. 
We could not determine a closed form expression for a Gr{\"o}bner basis for the ideal 
$\hilbertideal{m}= (\sigma_i(y+y^{-1}))$ for arbitrary $n \in \N$, 
although the reduced Gr{\"o}bner basis with respect to the lexicographical ordering of 
the ideal $\hilbertideal{a}=(\sigma_i(Y))$ is known, see \cite{Mora_Sala_Gröbner_basis}. 

\begin{question}
For type $\RootC$, what is the Gr{\"o}bner basis of the ideal $\hilbertideal{m}= (\sigma_i(y+y^{-1}))$ with respect to the lexicographical ordering?
Does the choice of monomial ordering matter in this context?
\end{question}

Analogous questions can be asked for other types next to $\RootC$, 
but we think that this would be the most accessible one, 
due to the relations in \cite[\S6]{metzlaff2024} that turn out to be more complicated for other types. 

In \Cref{thm:isotypic decomposition wrt filtration}, 
we give a formula for the multiplicities of the irreducible representations in $\CWeights_d$ for $\RootA[2]$, 
where the degree $d$ is defined via dilation of the Vorono{\"i} cell. 

\begin{question}
Are there explicit formulae for the multiplicities of the irreducible representations in the isotypic decomposition of $\CWeights_d$ for type $\RootA$, $n\geq 4$, or for the other types $\RootB,\,\RootC,\,\RootD,\,\RootE[6],\,\RootE[7],\,\RootE[8],\,\RootF[4]$? 
\end{question}

Theoretical generalizations beyond Weyl groups are foreseen to be possible under the following considerations. 
The setup in \Cref{sec_associated_graded_equivariants} only involves general finite groups that leave a lattice stable. 
In order to prove \Cref{prop_GradedCoinvariantIso}, 
we require the additive and multiplicative Hilbert ideal 
to be generated by invariants in the respective kernels 
$\augmentationideal{a}$ and $\augmentationideal{m}$. 
The proof of \Cref{thm_CoinvariantTranserAdditiveMultiplicative} relies on the fact 
that the coinvariant space affords the regular representation. 
For additive actions, 
this property characterizes complex reflection groups \cite{chevalley1955}. 
One should consider the associated complex root lattices in \cite{Nebe99}. 

Finally, a thrilling direction for future research lies in an extension of the study of multiplicative coinvariant spaces for diagonal actions, where one considers the action on $\C[\Weights\oplus\Weights^\vee]$. 
In the additive settings, foundational work on diagonal invariants and coinvariant spaces was developed by Haiman \cite{haiman1994} and Gordon \cite{gordon2003}. It would also be very interesting to explore the setting of diagonal actions with commuting and anticommuting variables as for instance studied in the additive setting in \cite{swanson}. 